\documentclass[11pt, oneside]{amsart}   	
\usepackage{amsmath,amsthm,amscd,amsfonts, amssymb, mathrsfs}
\usepackage{geometry}                		
\usepackage{graphicx}				
\usepackage{amssymb}
\newcommand{\sect}[1]{\section{#1}\setcounter{equation}{0}}

\font\mbn=msbm10 scaled \magstep1
\font\mbs=msbm7 scaled \magstep1
\font\mbss=msbm5 scaled \magstep1
\newfam\mbff
\textfont\mbff=\mbn
\scriptfont\mbff=\mbs
\scriptscriptfont\mbff=\mbss   

\newcommand{\RR}       { \mathbb{R}}

\newcommand{\N}       { \mathbb{N}}
\newcommand{\Z}        {\mathbb{Z}  }

\newtheorem{Th}{Theorem}[section]
\newtheorem{Lm}[Th]{Lemma}
\newtheorem{C}[Th]{Corollary}
\newtheorem{D}[Th]{Definition}
\newtheorem{St}[Th]{Stipulation}
\newtheorem{Stat}[Th]{Statement}
\newtheorem{Prop}[Th]{Proposition}
\newtheorem{R}[Th]{Remark}

\newtheorem{Notation}[Th]{Notation}

\newtheorem*{Theorem}{Theorem}
\begin{document}

\title[On Banach Structure of Multivariate BV spaces I]{On Banach Structure of Multivariate BV spaces I}
\author{Alexander Brudnyi}
\address{Department of Mathematics and Statistics\newline
\hspace*{1em} University of Calgary\newline
\hspace*{1em} Calgary, Alberta, Canada\newline
\hspace*{1em} T2N 1N4}
\email{abrudnyi@ucalgary.ca}
\author{Yuri Brudnyi}
\address{Department of Mathematics\newline
\hspace*{1em} Technion\newline
\hspace*{1em} Haifa, Israel\newline
\hspace*{1em} 32000}
\email{ybrudnyi@math.technion.ac.il}


\thanks{Research of the first author is supported in part by NSERC}

\begin{abstract}
We introduce and study multivariate generalizations  of the classical $BV$ spaces of Jordan, F. Riesz and Wiener. The family of the introduced spaces contains or is intimately related to a considerable class of function spaces of modern analysis including BMO, BV, Morrey spaces and those of Sobolev of arbitrary smoothness, Besov and Triebel-Lizorkin spaces.
We prove under mild restrictions that the BV spaces of this family are dual and present constructive characterizations of  their preduals via atomic decompositions. Moreover, we show that under additional restrictions such a predual space is isometrically isomorphic to the dual space of the separable subspace of the related $BV$ space generated by $C^\infty$ functions. As a corollary we obtain  the ``two stars theorem'' asserting that the second dual of this separable subspace is isometrically isomorphic to the $BV$ space. An essential role in the proofs play approximation properties of the $BV$ spaces under consideration, in particular, weak$^*$ denseness of their subspaces of $C^\infty$ functions.
Our results imply the similar ones (old and new) for the classical function spaces listed above obtained by the unified approach.
\end{abstract}

\date{}

\maketitle

\sect{Introduction}
\subsection{} Important properties of functions of bounded (Jordan) variation and their numerous applications in analysis have been attracting many researchers to define and study their multivariate analogs (Vitaly, Hardy, Lebesgue, Frechet, Tonelli, Kronrod, De Giorgi to name but a few). Each of the proposed definitions was directed to a multivariate generalization of a specific property of univariate $BV$ functions while those introduced in that way possessed (sometimes in disguise) also certain other important properties. For instance, the Hardy variation was introduced initially to generalize the Dirichlet convergence criterion to multivariate Fourier series but later it was discovered that measurable functions of bounded Hardy variation are in addition differentiable almost everywhere.

Another example is the Tonelli variation introduced initially for solving the problem  of the characterization of multivariate continuous functions with rectifiable graphs posed by Poincar\'{e}. However, at present the modern form of the Tonelli variation given successively by Cesari (1936), Fichera (1954) and De Giorgi (1954) plays an essential role in variational calculus and quasilinear PDEs of the first order, see, e.g., \cite{Gi-84} and \cite{AFP-00} for the results and the corresponding references.

As in the previous cases, the concept of variation presented below is intimately related to a specific problem of multivariate analysis,
the problem of the characterization of images of Sobolev spaces under continuous embeddings in certain spaces of integrable or continuous functions. The connection of the problem with the (Jordan) $BV$ spaces was  discovered  by Lebesgue \cite{Le-04} and Vitali \cite{Vi-05}. The corresponding result named the fundamental theorem of calculus, see, e.g., \cite[Ch.\,7]{Ru-87}, implies (in fact, is equivalent to) the next assertion.

\begin{Theorem}[Lebesgue] There is a linear isometry of $\dot W_1^1(0,1)$ in $BV[0,1]$ whose image  
denoted by $AC[0,1]$ consists of absolutely continuous functions.
\end{Theorem}

Hereafter $\dot W_p^k(\Omega)$, $1\le p\le\infty$, $k\in\N$, where $\Omega\subset\RR^d$ is a domain, stands for the homogeneous Sobolev space defined by a seminorm given for $f\in L_p(\Omega)$ by 
\begin{equation}\label{sobolev}
|f|_{W_p^k(\Omega)}:=\sum_{|\alpha|=k}\|D^\alpha f\|_{L_p(\Omega)}.
\end{equation} 

The result was extended to the space $\dot W_p^1(0,1)\hookrightarrow C[0,1]$, $1<p<\infty$, by F.\,Riesz \cite{Ri-10}. In this case, the image of the isometry coincides with the space $BV_p^{1/p'}\![0,1]$, $\frac{1}{p'}+\frac 1 p =1$, of functions of bounded $\bigl(\frac{1}{p'},p\bigr)$-variation. 

\noindent Here $(\lambda,p)$-variation of a function $f\in\ell^\infty[0,1]$ is given by 
\begin{equation}\label{var1}
var_p^\lambda\, f:=\sup_{\{x_i\}}\left( \sum_i\left(\frac{|f(x_{i+1})-f(x_i)|}{|x_{i+1}-x_i|^\lambda}  \right)^p   \right)^{\frac 1 p},
\end{equation}
where $\{x_i\}\subset [0,1]$ runs over monotone sequences.

Let us note that $BV_1^0$ is Jordan's space $BV$ and  $BV_p^0$ is the  Wiener-L.\,Young space $BV_p$.

The multivariate generalization of the above formulated results requires a new concept of variation that will be presented in the next subsection. The solution of the Sobolev embedding problem (in a sense, sharpening of the Sobolev embedding theorem) was given in \cite{Br-71} and is formulated in Subsection~1.3.4.

The following appropriately reformulated definition \eqref{var1} can be seen as a model case for the presented below concept of variation.
Actually, it is readily seen that
\begin{equation}\label{var2}
var_p^\lambda\, f=2\sup_{\pi}\left(\sum_{I\in\pi}\left(\frac{E_1(f;I)}{|I|^\lambda}\right)^p\right)^{\frac 1 p},
\end{equation}
where $\pi$ runs over families of nonoverlapping\footnote{i.e., with pairwise nonintersecting interiors} closed intervals $I\subset [0,1]$ and 
\begin{equation}\label{osc1}
E_1(f;I):=\inf_{c\in\RR}\sup_{I}|f-c|\, \left(=\frac 1 2 osc(f;I)\right).
\end{equation}
Replacing here the underlying space $\ell_\infty[0,1]$ by $L_q([0,1]^d)$, $1\le q\le\infty$, the families $\pi$ by those of nonoverlapping closed subcubes in $[0,1]^d$ and taking instead of constants  polynomials in $x\in\RR^d$ of a fixed degree we arrive to the required concept of variation.

The multiparametric family of $BV$ spaces defined by this variation includes or is intimately related to a considerable class of function spaces of modern analysis including $BMO$, $BV$, Morrey spaces and those of Sobolev of arbitrary smoothness, Besov and Triebel-Lizorkin spaces. In turn, the variational representation of the named spaces allows one to study them by a new approach combining tools of geometric analysis and approximation theory. The results obtained in this way for
Sobolev type embeddings, pointwise differentiability, Lusin type approximation, the real interpolation and nonlinear $n$-term approximation are presented in the survey \cite{Br-09}.

In the present paper, this approach amplified by tools of functional analysis is used to study the Banach structure of the $BV$ spaces introduced (duality, weak$^*$ compactness, two stars theorems etc.). These results imply  the similar ones  (old and new) for the classical function spaces  obtained by the unified approach.
\subsection{} An important ingredient of the  forthcoming definition of the variation is the following notion.
\begin{D}\label{equ1.2}
Local polynomial approximation of a function $f\in L_q^{\rm loc}(\RR^d)$, $1\le q\le\infty$, is a set function given for a bounded measurable set $S\subset\RR^d$ by
\begin{equation}\label{equa1.1}
E_{kq}(f;S):=\inf_{m\in\mathcal P_{k-1}^d}\|f-m\|_{L_q(S)},
\end{equation}
where $\mathcal P_\ell^d$ is the space of polynomials in $x=(x_1,\dots, x_d)\in\RR^d$ of degree $\ell$.
\end{D}
A geometric ingredient of the basic definition is the set of {\em packings} in $Q^d$ denoted by $\Pi(Q^d)$. Each packing consists of a finite family of pairwise nonoverlapping subcubes of $Q^d:=[0,1]^d$ homothetic to $Q^d$;  in what follows, {\em packings} are denoted by $\pi,\pi',\pi_i$, etc.
\begin{D}\label{def1.1}
Let $k\in\N$, $\lambda\in\RR$ and $1\le p,q\le\infty$. A function space $\dot{V}_{pq}^{k\lambda}(Q^d)$ is defined by a seminorm given for $f\in L_q(Q^d)$ by
\begin{equation}\label{equ1.3}
|f|_{V_{pq}^{k\lambda}}:=\sup_{\pi\in\Pi(Q^d)}\left\{\sum_{Q\in\pi}\left(|Q|^{-\lambda}E_{kq}(f;Q)\right)^p\right\}^{\frac 1 p};
\end{equation}
hereafter $|S|$ stands for the $d$-measure of a set $S\subset\RR^d$.
\end{D}
It can be easily verified that $\mathcal P_{k-1}^d|_{Q^d}$ is the null-space of $\dot{V}_{pq}^{k\lambda}(Q^d)$. Hence, the factor-space
\begin{equation}\label{equ1.4}
V_{pq}^{k\lambda}(Q^d):=\dot{V}_{pq}^{k\lambda}(Q^d)/\mathcal P_{k-1}^d|_{Q^d}
\end{equation}
is normed and \eqref{equ1.3} gives rise to its norm denoted by $\|\cdot\|_{V_{pq}^{k\lambda}}$.

The standard argument proves that $V_{pq}^{k\lambda}(Q^d)$ is a Banach space.

To simplify the notations, we set
\begin{equation}\label{equ1.5}
\kappa:=\{k,d,\lambda,p,q\}
\end{equation}
and write
\begin{equation}\label{equ1.6}
\dot{V}_\kappa:=\dot V_{pq}^{k\lambda}(Q^d),\qquad |\cdot|_\kappa:=|\cdot|_{V_{pq}^{k\lambda}(Q^d)},
\end{equation}
and similarly write $V_\kappa$ and $\|\cdot\|_\kappa$ for the corresponding factor-space and its norm.

In the sequel, the following separable subspaces of $\dot V_\kappa$ and $V_\kappa$ denoted by $\dot{\textsc{v}}_\kappa$ and $\textsc{v}_\kappa$ play an essential role:
\begin{equation}\label{equ1.7}
\dot{\textsc{v}}_\kappa:={\rm clos}(C^\infty\cap\dot V_\kappa,\dot V_\kappa),\qquad  \textsc{v}_\kappa:=\dot{\textsc{v}}_\kappa/\mathcal P_{k-1}^d;
\end{equation}
hereafter $C^\infty$ and $\mathcal P_{k-1}^d$ denote the following trace-spaces
\begin{equation}
C^\infty:=C^\infty(\RR^d)|_{Q^d}\quad {\rm and}\quad \mathcal P_{k-1}^d:=\mathcal P_{k-1}^d|_{Q^d}.
\end{equation}
(It will be shown that either $C^\infty\subset\dot V_\kappa$ or $\dot V_\kappa=\mathcal P_{k-1}^d$, i.e., $C^\infty\cap\dot V_\kappa$ is either $C^\infty$ or $\mathcal P_{k-1}^d$.)
\begin{St}\label{stip1.2}
{\rm Throughout the paper we fix the unit cube $Q^d$ and integer $k\ge 1$ removing them from the related symbols. For instance, we write $\kappa:=\{\lambda,p,q\}$ instead of that in \eqref{equ1.5} and $L_q$ instead of $L_q(Q^d)$. Moreover, we write $\lambda(\kappa), p(\kappa)$, etc. if these belong to $\kappa$.
However, these indices will be preserved if they assume other values, for instance, we write $V_{p\infty}^{1 0}[0,1]$ instead of $V_\kappa$ with $\kappa=\{1,1,0,p,\infty\}$.
}
\end{St}
Let us note that  local approximation $E_{kq}(f;Q)$ is equivalent to the $k$-{\em oscillation} of $f$ on $Q$ given by
\begin{equation}\label{equa1.8}
osc_{kq}(f;Q):=\sup_{h\in\RR^d}\|\Delta_h^k f\|_{L_q(Q_{kh})},
\end{equation}
where 
\[
\Delta_h^k:=\sum_{j=0}^k (-1)^{k-j} {k \choose j}\delta_{kh}
\]
and
\[
Q_{kh}:=\{x\in Q\, :\, x+kh\in Q\}.
\]
Namely, the next two-sided inequality with the constants of equivalence  depending only on $k,d$ is true, see \cite{Br-70},
\begin{equation}\label{equa1.9}
E_{kq}(f;Q)\approx osc_{kq}(f;Q).
\end{equation}
Hence, as in the one variable theory, $\dot V_\kappa$ functions can be equivalently defined by the behaviour of their oscillations.

In the classification of $V_\kappa$ spaces, the following  characteristic will be of essence.
\begin{D}\label{def1.3}
Smoothness of the space $V_\kappa$ denoted by $s(\kappa)$ is given by
\begin{equation}\label{equ1.9}
s(\kappa):=d\left(\lambda+\frac 1p -\frac 1q\right).
\end{equation}
\end{D}
Along with $p=p(\kappa)$ smoothness is invariant under linear isomorphisms of $V_\kappa$ spaces. Moreover, it is closely related to the differentiability and approximation characteristics of $\dot V_\kappa$ functions, see the survey \cite{Br-09}. For instance, a function of smoothness $s$ belongs for almost all $x\in Q^d$ to the Taylor class
$T_q^s(x)$ if $0<s<k$ and $t_q^k(x)$ if $s=k$, see, e.g., \cite[Sec.\,3.5]{Zi-87} for their definitions.

Let us finally note that the classical $BV$ spaces are defined over the space $\ell_\infty[0,1]$ of functions bounded on $[0,1]$ while $V_\kappa$ spaces with $q=\infty$ are defined over $L_\infty$ space. 
To include them and similar spaces in our consideration we use a version of Definition \ref{def1.1} with
local approximation denoted by $E_k(\cdot ;\cdot)$ that is defined for $f\in \ell_\infty^{\rm loc}(\RR^d)$ and $S\subset\RR^d$ by
\begin{equation}\label{equa1.10}
E_k(f;S):=\inf_{m\in\mathcal P_{k-1}^d}\sup_S |f-m|;
\end{equation}
the corresponding versions of the spaces $\dot V_\kappa$, $V_\kappa$, $\dot{\textsc{v}}_\kappa$ and $\textsc{v}_\kappa$ with $\kappa:=\{\lambda,p,\infty\}$ based on this definition are denoted by  $\dot V_p^\lambda$, $V_p^\lambda$, $\dot{\textsc{v}}_p^\lambda$ and $\textsc{v}_p^\lambda$, respectively ($k$ and $Q^d$ are omitted here by Stipulation \ref{stip1.2}).

In more details, $\dot V_p^\lambda$ is defined by a seminorm
\begin{equation}\label{equat2.10}
|f|_{V_p^\lambda}:=\sup_{\pi\in\Pi}\left(\sum_{Q\in\pi}\left(|Q|^{-\lambda} E_k(f;Q)\right)^p\right)^{\frac 1 p}
\end{equation}
and $V_p^\lambda:=\dot V_p^\lambda/{\mathcal P_{k-1}^d}$.

Moreover, smoothness of this space denoted by $s(\lambda,p)$ is defined by \eqref{equ1.9} with $q=\infty$.
\begin{R}\label{remark1.5}
{\rm It seems to be natural to identify $\dot V_p^\lambda$ with $\dot V_{\{\lambda,p,\infty\}}$ by choosing for each {\em class} $f$ from the latter space its representative, say, $\hat f\in\ell_\infty(Q^d)$. Unfortunately, this is impossible as the map $f\mapsto\hat f$ is not linear in general and does not preserve local approximation. Nevertheless, as it will be shown in the forthcoming paper such identification is possible for spaces $\dot V_p^\lambda$ and 
$\dot V_{\{\lambda,p,\infty\}}$ with $\lambda>0$ and $1\le p\le\infty$.
}
\end{R}

\subsection{}
Now we enumerate the classical function spaces that coincide with or are intimately related to the $V_\kappa$ spaces introduced.
\subsubsection{\underline{Consistency with the one variable definitions}} \hfill \smallskip

We begin with the space  $V_p^\lambda [0,1]$ whose associated seminorm is given for $f\in \ell_\infty[0,1]$ by
\begin{equation}\label{equat1.11}
|f|_{V_p^\lambda [0,1]}:=\sup_{\pi}\left\{\sum_{I\in\pi}\left(\frac{E_k(f;I)}{|I|^\lambda}\right)^p\right\}^{\frac 1 p};
\end{equation}
here $\pi$ runs over packings consisting of pairwise 
nonoverlapping  subintervals in $[0,1]$ and local approximation is given by \eqref{equa1.10}.

In turn, in the classical definitions, the supremum in \eqref{equat1.11} is taken over {\em coverings} of $[0,1]$ by nonoverlapping intervals and $E_k$ is replaced by the $k$-{\em deviation} $\delta_k$ given for $f\in\ell_\infty[0,1]$ and $I:=[a,b]\subset [0,1]$ by
\[
\delta_k(f;I):=|\Delta_h^k f(a)|,\quad {\rm where}\quad h:=\frac{b-a}{k};
\]
in particular, $\delta_1(f;I):=|f(b)-f(a)|$.

Equivalence of $V_p^\lambda[0,1]$ with the space obtained by these substitutions follows from 
the Whitney inequality \cite{Wh-59}
\[
E_k(f;I)\approx\sup_{I'\subset I}\delta_k(f;I),
\]
where the constants of equivalence are independent of 
$f$ and $I$ (note that $f$ here can be nonmeasurable).

In particular, $E_1(f;I)=\frac 1 2 \sup_{x,y\in I}|f(x)-f(y)|$;
therefore for $k=1$
\[
|f|_{V_p^\lambda [0,1]}=2^{-\frac 1 p}\sup_{x_i}\left(\sum_i\left(\frac{|f(x_{i+1})-f(x_i)|}{|x_{i+1}-x_i|^\lambda}\right)^p\right)^{\frac 1 p},
\]
where $\{x_i\}$ runs over finite monotone sequences in $[0,1]$.

The supremum here denoted by $var_p^\lambda(f)$, see \eqref{var2}, defines a seminormed space of functions on $[0,1]$ that we denote by $BV_p^\lambda$.

Thus, seminormed spaces $BV_p^\lambda$ and $\dot V_p^{1,\lambda}$ are isometrically isomorphic; in particular, the classical spaces of Jordan ($p=1$, $\lambda=0$), Wiener-L.\,Young ($1\le p<\infty$, $\lambda=0$) and F. Riesz ($1<p<\infty$, $\lambda=\frac{1}{p'}:=1-\frac 1 p $) are isometrically isomorphic
to the corresponding spaces $\dot V_p^{1,\lambda}$.
\subsubsection{
\underline{$V_\kappa$ spaces of negative smoothness} } $\,$ \smallskip

The $V_\kappa$ spaces with $s(\kappa)<0$ are closely related to weighted $L_p$ spaces  with nonintegrable singularities, Morrey spaces and the likes. In particular, Morrey space $M_q^s(Q^d)$, $1\le q<\infty$, $0<s<\frac{d}{q}$, is defined by a norm given for $f\in L_q(Q^d)$ by
\begin{equation}\label{equa1.17}
\|f\|_{ M_q^s}:=\sup_{Q\subset Q^d}|Q|^{\frac s d}\left(\frac{1}{|Q|}\int_Q|f|^q\, dx\right)^{\frac 1 q }.
\end{equation}

In spite of simplicity of the definition, Morrey spaces have numerous applications in PDEs and harmonic analysis, see, e.g., \cite{Ta-92} and \cite{AX-12} and references therein.

The relation to the space $ V_\kappa$ given by the equality\footnote{Hereafter $X=Y$ for (semi-)\,normed $X,Y$ means that they coincide as linear spaces and have equivalent (semi-)\,norms.}
\begin{equation}\label{equa1.18}
V_\kappa= M_q^s/\mathcal P_{k-1}^d,\quad {\rm where}\quad \kappa:=\left\{\frac 1 q -\frac s d ,\infty,q\right\},
\end{equation}
follows from  the inequality, see \cite{Ca-64},
\[
\|f\|_{M_q^s}\le c\left(\sup_{Q\subset Q^d}|Q|^{\frac s  d -\frac 1 q}E_{kq}(f;Q)+\|f\|_q\right).
\]

Let us note that here $ s(\kappa)=-s  < 0$.\smallskip

\subsubsection{\underline{$V_\kappa$ spaces of smoothness zero}}
$\,$\smallskip

Let $p,q,\lambda\in\kappa$ satisfy
\[
1\le q\le p<\infty,\quad \lambda=\frac 1 q -\frac 1 p ,
\]
hence, $s(\kappa)=0$.

Then for $q<p$
\begin{equation}\label{equa1.19}
L_q/\mathcal P_{k-1}^d\subsetneq V_\kappa\subsetneq L_{q\infty}/\mathcal P_{k-1}^d,
\end{equation}
and for $q=p$
\begin{equation}\label{equa1.20}
L_q/\mathcal P_{k-1}^d=V_\kappa\ ({\rm isometry}).
\end{equation}

Further,  if $\kappa:=\{\lambda, p, q\}$ and $k$ satisfy
\[
\lambda=1-\frac 1 p,\quad 1< p\le \infty,\quad q=1
\quad {\rm and}\quad k=1,
\]
i.e., $s(\kappa)=0$, then the space $\dot V_\kappa$ equals up to equivalence of the seminorms to the John-Nirenberg \cite{JN-61} space $BMO_p$ defined by a seminorm given for $f\in L_1$ by
\begin{equation}\label{equat1.18}
|f|_{BMO_p}:=\sup_{\pi\in\Pi}\left(\sum_{Q\in\pi}|Q|\left(\frac{1}{|Q|}\int_Q|f-f_Q|\,dx\right)^p\right)^{\frac 1 p};
\end{equation}
here $ f_Q:=\frac{1}{|Q|}\int_Q f\, dx$.

Denoting the expression under supremum by $\gamma(\pi;f)$ we define a subspace of $BMO_p$ denoted by $VMO_p$ by the condition
\begin{equation}
\lim_{\varepsilon\rightarrow\infty}\sup_{|\pi|\le\varepsilon}\gamma(\pi;f)=0,
\end{equation}
where $|\pi|:=\sup_{Q\in\pi}|Q|$.

The mostly used spaces with $p=\infty$ are denoted by $BMO, VMO$; the latter was introduced and studied for $d=1$ in \cite{Sa-75}. Numerous applications of these spaces in analysis are summarized in the book \cite{St-93}.

In general, we have for $k\ge 1$ and $1\le q<p\le\infty$, $\lambda=\frac 1 q-\frac 1 p $ the equality
\begin{equation}\label{equa1.21}
V_\kappa=BMO_p/\mathcal P_{k-1}^d\quad {\rm and}\quad \textsc{v}_\kappa=VMO_p/\mathcal P_{k-1}^d.
\end{equation}

\subsubsection{\underline{$V_\kappa$ spaces of positive smoothness} } $\,$\smallskip

(a) First, let $p,q,\lambda\in\kappa$ and $s(\kappa)$ satisfy
\begin{equation}\label{equa1.23}
1\le p<q<\infty,\quad \lambda=0,\quad s(\kappa)=k.
\end{equation}
Then it is true that
\begin{equation}\label{equa1.24}
\dot W_p^k=\dot{\textsc{v}}_\kappa.
\end{equation}
If $\lambda>0$, this equality holds also for $q=\infty$.

Now  let $d, p,q\in\kappa$ and $s(\kappa)$ be such that
\begin{equation}\label{equa1.25}
d\ge 2,\quad 1=p\le q\le\frac{d}{d-k}<\infty\quad {\rm and}\quad s(\kappa)=k.
\end{equation}
Then it is true that
\begin{equation}\label{equa1.26}
BV^k=\dot V_\kappa,
\end{equation}
where $BV^k$ consists of $L_1$ functions whose $k$-th distributional derivatives are finite  Borel measures on $Q^d$, see \cite[\S 4, Thm.\,12]{Br-71}.

Hence, a seminorm of $BV^k$ is given for $f\in L_1$ by
\begin{equation}\label{eq1.26a}
|f|_{BV^k}:=\sum_{|\alpha|=k}{{\rm var}}\,D^\alpha f\, \left(:=\sum_{|\alpha|=k}\|D^\alpha f\|_M\right).
\end{equation}

For $k=1$, this gives the seminorm of the classical space $BV(Q^d)$, see, e.g., the books \cite{Gi-84}, \cite{AFP-00} for properties and numerous applications of this space in analysis.

In turn, \eqref{equa1.26} implies the series of equivalent definitions of $BV$ in the spirit of that of Jordan. In fact,  in the notation of
\eqref{equat1.18}, the relation \eqref{equa1.26} gives for $f\in L_1$ 
\[
|f|_{BV(Q^d)}\approx \sup_{\pi\in\Pi (Q^d)}\sum_{Q\in\pi}\frac{\|f-f_Q\|_{L_q(Q)}}{|Q|^\lambda},
\]
where $1\le q\le\frac{d}{d-1}\, (<\infty)$, $\lambda=\frac{d-1}{d}-\frac 1 q $ and the constants of equivalence depend only on $d$.
\begin{R}\label{remark1.7}
{\rm (1) The case $s(\kappa)=k$ is maximal, since $\dot V_\kappa=\mathcal P_{k-1}^d$ if $s(\kappa)>k$, see Lemma \ref{lem3.1} below.

\noindent (2) Conditions \eqref{equa1.23}, \eqref{equa1.25} imply continuous embeddings of the corresponding  spaces in $L_q$. Moreover, if $\lambda>0$ these embeddings are compact.
}
\end{R}

(b) Finally, we consider relations of $V_\kappa$ spaces of smoothness 
\[
0<s:=s(\kappa)<k,
\]
to the homogeneous Besov (Lipschitz) spaces $\dot B_p^{sp}$ and $\dot B_p^{s\infty}$. The various applications of these spaces are surveyed in \cite{Tr-92}.

Let us recall that the space $\dot B_p^{s\theta}$, $1\le p,\theta\le\infty$, is defined by one of equivalent seminorms given for $f\in L_p(Q^d)$, an integer $0\le \ell< s$ and $k=k(s):=\min\{n\in\N\, :\, n>s\}$ by
\begin{equation}\label{equat1.24}
|f|_{B_p^{s\theta}}:=\sup_{|\alpha|=\ell}\left\{\int_0^1\left(\frac{\omega_{k-\ell,\, p}(D^\alpha f;t)}{t^{s-\ell}}\right)^\theta\, \frac{dt}{t}\right\}^{\frac 1 \theta };
\end{equation}
here $\omega_{kp}(f;\cdot)$ is the $k$-th modulus of continuity of $f\in L_p$, given by, cf. \eqref{equa1.8}, 
\begin{equation}\label{equat1.25}
\omega_{kp}(f;t):=\sup_{\|h\|_\infty\le t}\left\{\|\Delta_h^k f\|_{L_p(Q^d_{kh})}\right\},
\end{equation}
where $Q^d_{kh}:=\{x\in Q^d\, :\, x+kh\in Q^d\}$, $\|h\|_\infty:=\max_{1\le i\le d}|h_i|$.

Now under the conditions 
\begin{equation}\label{equat1.26}
k=k(s),\quad 1\le p<q<\infty\quad {\rm and}\quad s>0
\end{equation}
on $k,p,q\in\kappa$ and $s:=s(\kappa)$ the following continuous embeddings are true
\begin{equation}\label{equat1.27}
\dot B_p^s:=\dot B_p^{sp}\subset \dot{\textsc{v}}_\kappa\subset \dot V_\kappa\subset \dot B_p^{s\infty}.
\end{equation}
For $q=\infty$ the right-hand side embedding remains to be true but that of the left-hand side is true for $\dot B_p^s$  replaced by $\dot B_p^{s1}$.

Moreover, for $p=\infty$, $q\le\infty$
 \begin{equation}\label{equa1.27}
\dot V_\kappa=\dot B_\infty^{s}.
\end{equation}
\begin{R}
{\rm (1) For $1<p\le 2$, the left embedding \eqref{equat1.27} can be sharpen by replacing $\dot B_p^s$ by the larger space $F_p^{s2}/\mathcal P_{k(s)-1}^d$; here $F_p^{s\theta}$ is the Triebel-Lizorkin space, see, e.g., \cite{Tr-92} for its definition.

\noindent (2) Using the real interpolation, see, e.g., \cite[Thm.\,6.4.3(1)]{BL-76} one can represent the space $\dot B_p^{s\infty}$ as an interpolating space of the couple $(\dot{\textsc{v}}_{\kappa_0},\dot{\textsc{v}}_{\kappa_1})$, where $\kappa_i:=\{s_i,p,q_i\}$, $1\le p<q_i$, $i=0,1$, and $s(\kappa_0)=s(\kappa_1)$. Under this conditions we have
\[
\dot B_p^{s\infty}=(\dot{\textsc{v}}_{\kappa_0},\dot{\textsc{v}}_{\kappa_1})_{\theta\infty},
\]
where $s=s(1-\theta)+s\theta$, $0<\theta<1$.
}
\end{R}

The paper is organized as follows.\smallskip

In Section 2, we define the predual to the space $V_\kappa$ denoted by $U_\kappa$ and that to $V_p^\lambda$ denoted by $U_p^\lambda$. Then we formulate the main results of the paper and some directly following applications to the classical spaces described in Subsections 1.3.1--1.3.4.

In Section 3, we prove two results on $C^\infty$ approximation of $V_\kappa$ functions formulated in Subsection~2.2. The first one plays an essential role in the proofs of our duality results while the second one provides an important  characterization of functions of the space $\dot{\textsc{v}}_\kappa$.

In Sections 4 and 5, we prove Theorem \ref{prop1.4} describing the basic properties of the space $U_\kappa$ and Theorem \ref{te1.11} asserting that
under mild restrictions on the parameters the spaces $U_\kappa^*$ and $V_\kappa$ are isometrically isomorphic.

Finally, in Section 6, we prove Theorem \ref{teo1.20} asserting that under some additional restrictions the spaces $\textsc{v}_\kappa^{*}$ and $U_\kappa$ are isometrically isomorphic.  Passing to duals in the obtained relation we get the ``two stars theorem'' stating that $\textsc{v}_{\kappa}^{**}$ and $V_\kappa$ are isometrically isomorphic.

\sect{Formulation of Main Results}
\subsection{Duality}
In the first part, we define and describe the basic properties of a Banach space 
that  under mild restrictions is a predual to the space $V_\kappa$ (recall that $\kappa=\{\lambda,p,q\}$, see Stipulation \ref{stip1.2}). We also briefly discuss here similar results for the spaces $V_p^\lambda$, $\lambda\ge 0$, see \eqref{equat2.10}, leaving the detail account to a forthcoming paper.

In the second part, we present two approximation results for functions of $V_\kappa$ spaces. The first one is essentially used in the proofs of the duality theorems while the second one in the applications concerning the function spaces presented in Subsections~1.3.1--1.3.4.

Finally, we formulate the applications and refer to known before special cases of the presented results.

In the forthcoming formulations, we use the following:
\begin{Notation}\label{equ2.5}
{\rm We write for linear (semi-)\,normed vector spaces
\begin{equation}\label{eq2.7}
X\hookrightarrow Y
\end{equation}
if there is a linear continuous injection of $X$ into $Y$, and replace $\hookrightarrow $ by $\subset$ if the injection embeds $X$ into $Y$ as a linear subspace.

Further, we say that these spaces are {\em isomorphic} and write
\begin{equation}\label{eq2.8}
X\cong Y
\end{equation}
if $X\hookrightarrow Y$ and $Y\hookrightarrow X$, and 
\[
X=Y
\] 
if, in addition, they coincide as linear spaces, hence, have equivalent (semi-)\,norms. 

Finally, spaces $X$ and $Y$  are said to be {\em isometrically isomorphic} if the injections in \eqref{eq2.8} are of norm $1$.
We write in this case 
\begin{equation}\label{eq2.8a}
X\equiv Y.
\end{equation}
}
\end{Notation}

\subsubsection{Space predual to $V_\kappa$}
The space under consideration denoted by $U_\kappa$ is constructed by using the following building blocks.
\begin{D}[$\kappa$-{\em atom}]\label{def2.1}
A function $a\in L_{q'}$ is said to be a $\kappa$-atom on a subcube $Q\subset Q^d$ if it satisfies the conditions
\begin{itemize}
\item[(i)] ${\rm supp}\, a\subset Q$;\smallskip
\item[(ii)] $\|a\|_{q'}\le |Q|^{-\lambda}$;\smallskip
\item[(iii)] $\displaystyle\int_{Q^d} x^\alpha a(x)\, dx=0$ for all $|\alpha|\le k-1$.
\end{itemize}
As above, $\kappa:=\{\lambda, p,q\}$ and $\frac 1q +\frac{1}{q'}=1$ for $1\le q\le\infty$.

The subject of the definition is denoted by $a_Q$.
\end{D}

Let us recall, see Stipulation \ref{stip1.2}, that $k$ and $Q^d$ are fixed and removed from almost all notations, e.g., $\|a\|_{q'}:=\|a\|_{L_{q'}(Q^d)}$.
\begin{D}[$\kappa$-chain]\label{def2.2}
A function $b\in L_{q'}$ is said to be a $\kappa$-chain subordinate to a packing $\pi\in\Pi$ if $b$ belong to the linear span of the family of $\kappa$-atoms $\{a_Q\}_{Q\in\pi}$.

The subject of this definition is denoted by $b_\pi$.
\end{D} 

Moreover, we write
\begin{equation}\label{eq2.1}
[b_\pi ]_{p'}:=\|\{c_Q\}_{Q\in\pi}\|_{p'}:=\left\{\sum_{Q\in\pi} |c_Q|^{p'}\right\}^{\frac{1}{p'}}
\end{equation}
whenever
\begin{equation}\label{eq2.2}
b_\pi=\sum_{Q\in\pi}c_Q\, a_Q.
\end{equation}
This clearly defines a norm on the linear span of the family $\{a_Q\}_{Q\in\pi}$.

Further, let $U_\kappa^0\subset L_{q'}$ denote the linear span of the set $\mathcal A_\kappa$ of all $\kappa$-atoms, i.e.,
\begin{equation}\label{eq2.3}
U_\kappa^0:={\rm linspan}\{a_Q\in \mathcal A_\kappa \}.
\end{equation}
Every $f\in U_\kappa^0$ can be represented (in infinitely many ways) as a finite sum of $\kappa$-chains by
\begin{equation}\label{eq2.4}
f=\sum_\pi b_\pi.
\end{equation}
The space $U_\kappa^0$ is equipped with the seminorm\footnote{in fact, we show that under mild restrictions on $\kappa$ \eqref{eq2.5} is a norm.}
\begin{equation}\label{eq2.5}
\|f\|_{U_\kappa^0}:=\inf\sum_\pi\, [b_\pi]_{p'},
\end{equation}
where infimum is taken over all representations \eqref{eq2.4}.
\begin{D}\label{def2.3}
The space $U_\kappa$ is the completion of the seminormed space $(U_\kappa^0,\|\cdot\|_{U_\kappa^0})$.
\end{D}

The next result describes the basic properties of the space $U_\kappa$.
\begin{Th}\label{prop1.4}
(a) The closed unit ball of $U_\kappa$ denoted by $B(U_\kappa)$
is the closure of the symmetric convex hull of the set
$\mathcal B_\kappa:=\{b_\pi\in U_\kappa^0\, :\, [b_\pi]_{p'}\le 1\}$.\smallskip

\noindent (b) If $p,q\in\kappa$ satisfy the conditions
\[
1<q\le\infty\quad {\rm and}\quad 1\le p\le\infty, 
\]
then $U_\kappa$ is separable.\smallskip

\noindent (c) If $p,q\in\kappa$ and $s:=s(\kappa)$ satisfy the conditions
\begin{equation}\label{eq2.6}
1<q\le\infty,\quad 1\le p\le\infty,\quad s\le k,
\end{equation}
then $U_\kappa$ is Banach.
\end{Th}
Now we present a duality theorem for the space $V_\kappa$. 

\begin{Th}\label{te1.11}
If $p,q\in\kappa$ and $s:=s(\kappa)$ satisfy conditions \eqref{eq2.6},
then 
\[
U_\kappa^*\equiv V_\kappa.
\]
More precisely, each continuous linear functional on $U_\kappa$ has the form
\[
f(u)=\int_{Q^d}fu\,dx\quad {\rm for\ all}\quad u\in U_\kappa^0\, (\subset L_{q'}),
\]
where $f\in V_\kappa\, (\subset L_q)$ and $\|f\|_{V_\kappa}$ is equal to the linear functional norm.
\end{Th}

\subsubsection{Two Stars Theorem}

Using the properties of $U_\kappa$ and $V_\kappa$ presented here and in the next subsection  and some basic facts of the Banach space theory we prove the following:
\begin{Th}\label{teo1.20}
Let $p,q\in\kappa$ and $s:=s(\kappa)$ satisfy the conditions
\begin{equation}\label{eq2.12}
1< p\le\infty ,\quad 1<q<\infty\quad {\rm and}\quad s<k.
\end{equation}
Then  
\[
\textsc{v}_\kappa^{*}\equiv U_\kappa.
\]
\end{Th}
Specifically, the result asserts that under the identification of $\textsc{v}_\kappa$ with a subspace of $U_\kappa^*$ by Theorem \ref{te1.11} each continuous linear functional on $\textsc{v}_\kappa$ has the form $T_u(f)=f(u)$ for all $f\in \textsc{v}_\kappa$, where $u\in U_\kappa$  and $\|u\|_{U_\kappa}$ is equal to the linear functional norm.

From here and Theorem \ref{te1.11} we obtain:
\begin{C}\label{cor2.8}
Under conditions \eqref{eq2.12}
\begin{equation}\label{eq2.18a}
 \textsc{v}_\kappa^{**}\equiv V_\kappa.
 \end{equation}
 \end{C}

\begin{R}
{\rm 
The restriction $q>1$ is necessary. In fact, the space $V_\kappa$ with $\kappa=\{0,1,1\}$, hence,
$s(\kappa)=0<k$, is isometrically isomorphic to the space $L_1/\mathcal P_{k-1}^d\cong L_1$, see \eqref{equa1.20}, that is not dual.

 The restriction $s(\kappa)<k$ is necessary as well, see Remark \ref{rem2.15} below.}
\end{R}

\subsubsection{Space predual to $V_p^\lambda$} Let us recall that $\dot V_p^\lambda$ is defined by seminorm \eqref{equat2.10},
\[
V_p^\lambda:=\dot V_p^\lambda/\mathcal P_{k-1}^d\quad {\rm and}\quad \textsc{v}_p^\lambda:={\rm clos}(C^\infty/\mathcal P_{k-1}^d,V_p^\lambda).
\]

To define a predual to $V_p^\lambda$ denoted by $U_p^\lambda$
we follow the scheme of Subsection~2.1.1 that begins with the definition of atoms. The required version is motivated by the equality $\ell_1^*=\ell_\infty$, where $\ell_1$ is defined by a norm given for $f:Q^d\rightarrow\RR$ by
\begin{equation}\label{equat2.13}
\|f\|_{\ell_1}:=\sum_{x\in Q^d}|f(x)|.
\end{equation}
The space $\ell_1$ is nonseparable but every $f\in\ell_1$ has at most countable support. This leads to the following:
\begin{D}\label{def2.4a}
A function $a_Q:Q^d\rightarrow\RR$ is said to be a $(\lambda, p)$-atom on $Q\subset Q^d$ if $a_Q$ is supported by $Q$ and satisfies the conditions
\begin{itemize}
\item[(i)] 
\[
\|a_Q\|_{\ell_1}\le |Q|^{-\lambda};
\]
\item[(ii)]
\[
 \sum_{x\in Q^d} m(x)a_Q(x)=0\quad {\rm for\ every}\quad m\in\mathcal P_{k-1}^d.
 \]
\end{itemize}
\end{D}
Having this we repeat word-for-word definitions of Subsection 2.1.1 to introduce $(\lambda,p)$-chains and their norms, see \eqref{eq2.1} and \eqref{eq2.2}, preserving the very same notations.

Further, $(U_p^\lambda)^0\subset\ell_1$ is the linear span of the set $\mathcal A_{\{\lambda,p\}}$ of all $(\lambda,p)$-atoms, i.e.,
\begin{equation}\label{equat2.14}
(U_p^\lambda)^0:={\rm linspan}\{a_Q\in \mathcal A_{\{\lambda,p\}}\}.
\end{equation}
As above, see \eqref{eq2.5}, this space is equipped with the seminorm $\|\cdot\|_{(U_p^\lambda)^0}$.

Finally, the completion of $(U_p^\lambda)^0$ under this seminorm gives the required space $U_p^\lambda$.

The basic result for this case asserts:
\begin{Th}\label{teor2.11}
(a)  Let the parameters of the space $V_p^\lambda$ and its smoothness $s:=d(\lambda+\frac 1 p)$ satisfy the conditions
\begin{equation}\label{e2.14}
\lambda\ge 0,\quad 1<p\le\infty\quad {\rm and}\quad s\le k.
\end{equation}
Then $(U_p^\lambda)^*\equiv V_p^\lambda$.\smallskip

\noindent (b) If $\lambda, p$ satisfy \eqref{e2.14} and $s<k$, then 
$(\textsc{v}_p^\lambda)^*\equiv U_p^\lambda$ if $\lambda>0$ and 
$(\textsc{v}_p^0)^*\cong U_p^0$.
\end{Th}

As a corollary we obtain the corresponding two stars theorem: $(\textsc{v}_p^\lambda)^{**}\equiv V_p^\lambda$ if $\lambda>0$ and  $(\textsc{v}_p^0)^{**}\cong V_p^0$.

\subsection{Approximation Theorems} We present here two results on $C^\infty$ approximation of $V_\kappa$ functions. 

The first result plays an essential role in the proofs of the duality theorems.
\begin{Th}\label{teo2.12}
(a) For each function $f\in\dot V_\kappa$, $\kappa:=\{\lambda,p,q\}$,  there is a sequence $\{f_n\}_{n\in\N}\subset C^\infty$ linearly depending on $f$ such that
\begin{equation}\label{e2.16}
\lim_{n\rightarrow\infty}|f_n|_{V_\kappa}=|f|_{V_\kappa}.
\end{equation}
(b) Moreover,
\begin{equation}\label{e2.17}
\lim_{n\rightarrow\infty}\|f-f_n\|_q=0\quad {\rm if}\quad 1\le q<\infty
\end{equation}
and
\begin{equation}\label{e2.18}
\lim_{n\rightarrow\infty}\int_{Q^d}(f-f_n)g\,dx=0\quad {\rm for\ each}\quad g\in L_1\quad {\rm if}\quad q=\infty.
\end{equation}
\end{Th}

The second result characterizes $\dot V_\kappa$ functions admitting $C^\infty$ approximation, i.e., functions of the subspace $\dot{\textsc{v}}_\kappa$.
\begin{Th}\label{teor2.13}
Let $f\in\dot V_\kappa$, where $\kappa:=\{\lambda, p, q\}$ satisfies one of the conditions 
\begin{equation}\label{e2.19} 
1\le p\le\infty,\quad 1\le q<\infty\quad {\rm and}\quad s(\kappa)<k,
\end{equation}
or
\begin{equation}\label{e2.19a}
1\le p\le\infty,\quad q=\infty,\quad \lambda\ge 0 \quad {\rm and}\quad s(\kappa)<k.
\end{equation} 
Then $f\in \dot{\textsc{v}}_\kappa$ if and only if 
\begin{equation}\label{e2.20}
\lim_{\varepsilon\rightarrow 0}\sup_{|\pi|\le\varepsilon}\left(\sum_{Q\in\pi}\left(|Q|^{-\lambda}E_{kq}(f;Q)\right)^p\right)^{\frac 1 p}=0;
\end{equation}
hereafter $|\pi|:=\sup_{Q\in\pi}|Q|$.
\end{Th}
\begin{R}
{\rm The restriction $s(\kappa)<k$ is necessary. In fact, for $s(\kappa)=s$ 
the subspace of functions of $\dot V_\kappa$ satisfying condition \eqref{e2.20} is (algebraically) isomorphic to $\mathcal P_{k-1}^d$, see the argument of the proof of Lemma \ref{lem3.1}, while $\dot{\textsc{v}}_\kappa$ contains all functions from $C^\infty\, (\subset L_q)$, see the proof of Theorem \ref{prop1.4}\,(c).
}
\end{R}
\subsection{Applications} We begin with the result describing duality properties of the ``classical'' spaces $\dot V_p^\lambda[0,1]$, $\lambda\ge 0$, $1\le p\le\infty$. It is a corollary of the above formulated Theorem \ref{teor2.11}.

In the following discussion, we use the classical definition of the $\dot V_p^\lambda[0,1]$ seminorm given for $f\in\ell_\infty[0,1]$ by
\begin{equation}\label{e2.21}
|f|_{V_p^\lambda}:=\sup_{\{x_i\}}\left(\sum_i\left(\frac{\delta_k(f;x_i,x_{i+1})}{|x_{i+1}-x_i|^\lambda}\right)^p\right)^{\frac 1 p},
\end{equation}
where $\{x_i\}$ runs over all monotone finite sequences in $[0,1]$ and
\[
\delta_k(f;a,b):=|\Delta_h^k(f;a)|,\quad {\rm where}\quad h:=\frac{b-a}{k}.
\]

Let us recall that
\[
V_p^\lambda:=\dot V_p^\lambda/\mathcal P_{k-1}^1\quad {\rm and}\quad \textsc{v}_p^\lambda:={\rm clos}(C^\infty/\mathcal P_{k-1}^1,V_p^\lambda).\smallskip
\]

\noindent {\bf A1. \underline{Duality Properties of $\dot V_p^\lambda[0,1]$}:}\medskip

 {\em Assume that
\begin{equation}\label{e2.22}
1<p\le\infty\quad {\rm and}\quad 0\le \lambda<k-\frac 1 p.
\end{equation}

Then }
\begin{equation}\label{e2.23}
(\textsc{v}_p^\lambda)^{**}\equiv V_p^\lambda\quad {\rm if}\quad \lambda>0\quad {\rm and}\quad (\textsc{v}_p^0)^{**}\cong V_p^0.\medskip
\end{equation}

Let us show that the restrictions on $\lambda, p$ in \eqref{e2.22} are necessary. In fact, the space $\dot V_p^\lambda[0,1]$ with $k=p=1$, $\lambda=0$, hence, $\lambda=k-\frac 1 p$, coincides with the Jordan space $BV[0,1]$. By the (modernized form of) Lebesgue theorem, see Subsection 1.1, 
\[
\dot W_1^1(0,1)\equiv AC[0,1].
\]
Since this isometry preserves $C^\infty$ functions and the subspace $C^\infty$ is dense in the Sobolev space, 
\[
\dot{\textsc{v}}_1^0:={\rm clos}(C^\infty, BV)=AC\equiv \dot W_1^1.
\]
Factorizing by constants we obtain 
\[
\textsc{v}_1^0\equiv W_1^1\cong \dot W_1^1/\RR\cong L_1;
\]
this, in turn, implies that
\[
(\textsc{v}_1^0)^{**}\cong L_\infty^*.
\]
Assuming that \eqref{e2.23} is true in this case we get
\[
BV\cong (\textsc{v}_1^0)^{**}\cong L_\infty^*.
\]
However, this is false as the cardinality of $BV$ is strictly less than that of $L_\infty^*$.

\begin{R}\label{rem2.15}
{\rm This also shows that the restriction on $\kappa$ in Theorem \ref{teo1.20} is necessary as well.
In fact, it will be proved in a forthcoming paper that for $\kappa=\{\lambda, p,\infty\}$ with $\lambda\ge 0$, $1\le p\le\infty$
\[
V_\kappa\hookrightarrow V_p^\lambda\quad {\rm and}\quad \textsc{v}_\kappa\equiv\textsc{v}_p^\lambda.
\]
If Theorem \ref{teo1.20} is valid for $k=d=p=1$, $\lambda=0$, hence, for $s(\kappa)=k$, then we have a contradiction
\[
L_\infty^*\cong (\textsc{v}_\kappa)^{**}\cong V_\kappa\hookrightarrow V_1^0=BV/\RR.
\]
}
\end{R}

We complete this discussion by referring to the papers \cite{DeL-61} and \cite{Ki-84}. The first one contains the two stars theorem \eqref{e2.23}
 for the space $Lip^\lambda$, $0<\lambda< 1$, that coincides with $V_\infty^\lambda[0,1]$ while the second one proves \eqref{e2.23} for the Wiener-L.\,Young space $BV_p[0,1]$, $1<p<\infty$, that coincides with $V_p^0[0,1]$.\smallskip
 
The subsequent applications present new results concerning approximation and duality properties of Morrey, John-Nirenberg, Sobolev and Lipschitz spaces. Their proofs directly follow from the formulated above results for $V_\kappa$ spaces via the corresponding isomorphisms between them and the spaces under consideration.\smallskip

Our first result uses the isomorphism for Morrey space $M_q^s$, see \eqref{equa1.18}, where we take $k=1$; hence, we have
\begin{equation}\label{e2.24}
M_q^s/\RR\cong V_\kappa,\quad s>0,
\end{equation}
where $\kappa:=\{\lambda, p, q\}$ and $k$ satisfy
\begin{equation}\label{e2.25}
k:=1,\quad \lambda:=\frac 1 q - \frac s d >0,\quad p=\infty,\quad 1\le q<\infty.\smallskip
\end{equation}
\noindent {\bf A2. \underline{Properties of Morrey Space}:}\medskip

\begin{itemize}
\item[(a)]
{\em  Let $f\in M_q^s$, where $s>0$ and $q$ satisfy \eqref{e2.25}.

\noindent There exists a sequence $\{f_n\}_{n\in\N}\subset C^\infty$ linearly depending on $f$   such that
\[
\lim_{n\rightarrow\infty} f_n=f\quad {\rm in}\quad L_q\quad {\rm and}\quad \lim_{n\rightarrow\infty}\|f_n\|_{M_q^s}\approx \|f\|_{M_q^s}
\]
with constants of equivalence depending on $d$ and $\lambda$ only.

\noindent Moreover, this sequence converges to $f$ in $M_q^s$ if and only if
\begin{equation}\label{e2.26}
\lim_{|Q|\rightarrow 0}|Q|^{\frac s d}\left(\frac{1}{|Q|}\int_Q|f-f_Q|^q\, dx\right)^{\frac 1 q}=0.
\end{equation}
\item[(b)] Let the parameters $k,\lambda, p$ of $\kappa$ satisfy condition \eqref{e2.25} while $1<q<\infty$. Then it is true that
\[
U_\kappa^*\cong M_q^s/\RR.
\]
Moreover, denoting by $m_q^s$ the space of functions $f\in L_q$ satisfying \eqref{e2.26} we have
\begin{equation}\label{e2.26a}
(m_q^s/\RR)^{*}\cong U_\kappa^*.
\end{equation}
In particular,}
\begin{equation}\label{e2.26b}
(m_q^s/\RR)^{**}\cong M_q^s/\RR.
\end{equation}
\end{itemize}
\begin{R}
{\rm Some other preduals to the space $M_q^s$ see in \cite{Zo-86} and \cite{AX-12}.}
\end{R}

To derive the next result we use isomorphism \eqref{equat1.18} with $k=1$. Hence, in this case we have
\[
BMO_p/\RR\cong V_\kappa,
\]
where $\kappa:=\{\lambda, p, q\}$ and $k$ satisfy
\begin{equation}\label{e2.27}
k:=1,\quad \lambda:=\frac 1 q -\frac 1 p,\quad 1\le q<p\le\infty.\medskip
\end{equation}

\noindent{\bf A3. \underline{Properties of $BMO$ Space}:}\medskip

\begin{itemize}
\item[(a)] {\em  Let $f\in BMO_p$ and $k,p,q$ satisfy \eqref{e2.27}.

\noindent There exists a sequence $\{f_n\}_{n\in\N}\subset C^\infty$ linearly depending on $f$   such that
\[
\lim_{n\rightarrow\infty} f_n=f\quad {\rm in}\quad L_q\quad {\rm and}\quad \lim_{n\rightarrow\infty}|f_n|_{BMO_p}\approx |f|_{BMO_p}
\]
with constants of equivalence depending on $d$ and $\lambda$ only.

\noindent Moreover, this sequence converges to $f$ in $BMO_p$ if and only if $f\in VMO_p$, i.e.,
\[
\lim_{\varepsilon\rightarrow\infty}\sup_{|\pi|\le\varepsilon}\left(\sum_{Q\in\pi}|Q|\left(\frac{1}{|Q|}\int_Q|f-f_Q|\,dx\right)^p\right)^{\frac 1 p}=0.
\]
\item[(b)] Let the parameters $k,\lambda, p$ of $\kappa$ satisfy condition \eqref{e2.27} but $1<q<\infty$. Then it is true that
\begin{equation}\label{e2.28}
U_\kappa^*\cong BMO_p/\RR
\end{equation}
and, moreover,
\begin{equation}\label{e2.29}
VMO_p^{*}\cong U_\kappa. 
\end{equation}
In particular,}
\begin{equation}\label{e2.30}
VMO_p^{**}\cong BMO_p.
\end{equation}
\end{itemize}
\begin{R}
{\rm 
(1) The special case $p=\infty$ of (b), i.e., the relations $U_{\kappa}^*\cong BMO/\RR$ and $VMO^{*}=U_\kappa$, $\kappa=\bigl\{\frac 1 q,\infty, q\bigr\}$, was proved in \cite{CW-77} in a more general setting of functions on homogeneous metric spaces.  The space $U_{\kappa}$ is denoted by $H^1$ in \cite{CW-77} to emphasize its connection with the atomic decomposition of the classical Hardy space $H^1(\RR^d)$ and the celebrated C. Fefferman duality theorem \cite{F-71} asserting that $BMO(\RR^d)=(H^1(\RR^d))^*$.

\noindent (2) Relation \eqref{e2.28} presents a family of preduals to the space $BMO_p$  which are pairwise isomorphic due to \eqref{e2.29}.}
\end{R}

The third result is derived from isomorphism \eqref{equa1.26}
\[
BV^k=\dot V_\kappa,
\]
where $\kappa=\{\lambda, p,q\}$ satisfies
\begin{equation}\label{e2.33}
 \lambda=\frac k d -\frac{1}{q'},\quad p=1,\quad 1\le q\le\infty,\end{equation}
 hence, $s(\kappa)=k$.\medskip

\noindent {\bf A4. \underline{Properties of $BV^k$ Space}:}\medskip

\begin{itemize}
\item[(a)] {\em Let $f\in BV^k$ and $\kappa$ satisfy \eqref{e2.33}.

\noindent There exists a sequence $\{f_n\}_{n\in\N}\subset C^\infty$ linearly depending on $f$ such that
\begin{equation}\label{e2.34}
\lim_{n\rightarrow\infty} f_n=f\quad {\rm in}\quad L_q\quad {\rm and}\quad \lim_{n\rightarrow\infty}|f_n|_{BV^k}\approx |f|_{BV^k}
\end{equation}
with constants of equivalence depending on $k, q$ and $\lambda$.

\noindent Moreover, the sequence converges to $f$ in $BV^k$ if and only if 
\begin{equation}\label{e2.35}
\lim_{\varepsilon\rightarrow\infty}\sup_{|\pi|\le\varepsilon}\left(\sum_{Q\in\pi}\frac{E_{kq}(f;Q)}{|Q|^\lambda}\right)=0.
\end{equation}
\item[(b)] If $\lambda,p\in\kappa$ satisfy \eqref{e2.33} and $1<q\le\infty$, then it is true that}
\begin{equation}\label{e2.36}
U_\kappa^*\cong BV^k/\mathcal P_{k-1}^d.
\end{equation}
\end{itemize}
\begin{R}
{\rm For $k=1$, i.e., for the space $BV(Q^d)$, statement \eqref{e2.34} is known, see, e.g, \cite[Thm.\,5.3.3]{Zi-87} while that of \eqref{e2.35} asserting that $BV(Q^d)$ has a predual space admitting an atomic decomposition is new. Duality of  $BV(Q^d)$ was mentioned without the proof in \cite[Remark 3.12]{AFP-00}.}
\end{R}

Finally, we use isomorphism \eqref{equa1.27} asserting that
\[
\dot V_\kappa=\dot B_\infty^{s}
\]
under the conditions on $k$ and $\kappa:=\{\lambda, p, q\}$ given by 
\begin{equation}\label{e2.37}
k=\min\{n\in\N\, :\, n>s\},\quad \lambda:=\frac s d+\frac 1 q,\quad p=\infty,\quad 1\le q\le \infty\quad {\rm and}\quad s>0.
\end{equation}
Since $\dot B_\infty^s$ consists of continuous up to Lebesgue measure zero functions, it is naturally identified with the space $\dot\Lambda^s(Q^d)$ of functions $f\in C(Q^d)$ satisfying the condition
\begin{equation}\label{e2.38}
\sup_{x,h}\bigl\{|\Delta_h^k(f;x)|\, :\, x, x+kh\in Q^d\bigr\}\le C\|h\|_\infty^s,
\end{equation}
where $k:=k(s)$.

\noindent This is equipped with a (Banach) seminorm given for $f$ satisfying \eqref{e2.38} by
\[
|f|_{\Lambda^s}:=\inf\, C.
\]

Further, replacing the right-hand side in \eqref{e2.38} by $o(\|h\|_\infty^s)$, $h\rightarrow 0$, we define the subspace of $\dot\Lambda^s$ denoted by $\dot\lambda^s$.

As usual, we also set
\[
\Lambda^s:=\dot\Lambda^s/\mathcal P_{k(s)-1}^d,\qquad  \lambda^s:=\dot\lambda^s/\mathcal P_{k-1}^d.
\]

\noindent {\bf A5. \underline{Properties of $\Lambda^s$ }:}\medskip

\begin{itemize}
\item[(a)] {\em Let $f\in \dot\Lambda^s$ and $k$ and  $\kappa$ satisfy \eqref{e2.37}.

\noindent There exists a sequence $\{f_n\}_{n\in\N}\subset C^\infty$ linearly depending on $f$ such that
\begin{equation}\label{e2.39}
\lim_{n\rightarrow\infty} f_n=f\quad {\rm in}\quad C\quad {\rm and}\quad \lim_{n\rightarrow\infty}|f_n|_{\Lambda^s}\approx |f|_{\Lambda^s}
\end{equation}
with constants of equivalence depending on $d$ and $s$ and tending to $\infty$ if $s\rightarrow k$.

\noindent Moreover, the sequence converges to $f$ in $\dot\Lambda^s$ if and only if it belongs to $\dot\lambda^s$.
\item[(b)] If $k$ and $\kappa$ satisfy \eqref{e2.37} but with $1<q<\infty$, then it is true that}
\begin{equation}\label{e2.40}
U_\kappa^*\cong \Lambda^s
\end{equation}
and, moreover,
\begin{equation}\label{e2.41}
(\lambda^s)^*\cong U_\kappa.
\end{equation}
In particular,
\begin{equation}\label{e2.42}
(\lambda^s)^{**}\cong\Lambda^s.
\end{equation}
\end{itemize}

As it was mentioned above, relation \eqref{e2.42} with $s\in (0,1)$ and $d=1$ was originally proved in \cite{DeL-61}. Similar to \eqref{e2.40}--\eqref{e2.42} statements for analogous spaces of functions on $\RR^d$ follow from the main results of \cite{Ha-97}.
\subsection{Comments}
(a) Using Definitions \ref{def1.1} and \ref{def2.1} with $\RR^d$ substituted for $Q^d$ and $f\in L_q(\RR^d)$ we define the family of spaces $V_\kappa(\RR^d)$ and $U_\kappa(\RR^d)$. 
The corresponding space $\textsc{v}_\kappa(\RR^d)$ is defined as ${\rm clos}(C_0^\infty(\RR^d),V_\kappa(\RR^d))$, where $C_0^\infty(\RR^d)$ consists of $C^\infty$ functions with compact supports.  Then under the restriction $\lambda(\kappa)>0$ all of the main results
of Subsection~2.2 hold with the very same proofs.

For $k(\kappa)=1$, the same is true for the family of spaces $\{V_\kappa\}$ of functions on complete doubling metric spaces, see, e.g., \cite[sec.\,4.3]{BB-11} for their definition and properties. In this cases, cubes are replaced by metric balls.\smallskip

\noindent (b) Let $\omega: [0,1]\rightarrow\RR_+$ be a continuous nondecreasing function satisfying the conditions
\begin{equation}
\lim_{\lambda\rightarrow 1^+}\left(\sup_{t\in (0,\lambda^{-1}]}\frac{\omega(\lambda t)}{\omega(t)}\right)=1\quad {\rm and}\quad \varlimsup_{t\rightarrow 0}\,\frac{t^{\sigma}}{\omega(t)}<\infty,
\end{equation}
where $\sigma:=\frac k d - \frac 1 p +\frac 1 q$.

Replacing in Definitions \ref{def1.1} and \ref{def2.1} $|Q|^\lambda$ by $\omega(|Q|)$ we introduce the spaces denoted by  $V_\kappa^\omega$ and $U_\kappa^\omega$, where $\kappa:=\{k,p,q\}$. The analogs of Theorems \ref{prop1.4}, \ref{te1.11} and \ref{teo2.12} hold for these spaces with minor changes in the proofs.

The same is true for  Theorems \ref{teo1.20}, \ref{teor2.13} and Corollary \ref{cor2.8} under one more condition on $\omega$:
\begin{equation}
\lim_{t\rightarrow 0}\, \frac{t^{\sigma}}{\omega(t)}=0.
\end{equation}

\sect{Proofs of Approximation Theorems}
\subsection{$BV$ spaces of large smoothness} The next result shows that $\dot V_\kappa=\mathcal P_{k-1}^d$ if $s(\kappa)>k$. In particular,  Theorem \ref{teo2.12} trivially holds in this case.
\begin{Lm}\label{lem3.1}
Let $f\in\dot V_\kappa$, where $\kappa:=\{\lambda, p,q\}$ and
$s:=s(\kappa)>k$. Then $f$ equals a  polynomial of degree $\le k-1$ a.e. on $Q^d$.
\end{Lm}
\begin{proof}
Setting $r:=\min\{p,q\}$ and using the H\"{o}lder inequality we have
\[
E_{kr}(f;Q)\le |Q|^{\frac 1 r -\frac 1 q}E_{kq}(f;Q).
\]
Since $\frac s d:=\lambda+\frac 1 p -\frac 1 q$, this and the H\"{o}lder inequality imply
\[
\begin{array}{l}
\displaystyle \left\{\sum_{Q\in\pi}\left(|Q|^{-\frac s d}E_{kr}(f;Q)\right)^r\right\}^{\frac 1 r}\le \left\{\sum_{Q\in\pi}\left(|Q|^{-\lambda}E_{kq}(f;Q)\right)^r |Q|^{1-\frac r p}\right\}^{\frac 1 r} \\
\\
\displaystyle \le \left\{\sum_{Q\in\pi}\left(|Q|^{-\lambda} E_{kq}(f;Q)\right)^p\right\}^{\frac 1 p}\left\{\sum_{Q\in\pi}|Q|\right\}^{\frac 1 r -\frac 1 p}\le |f|_{V_\kappa}.
\end{array}
\]

Now, let $\Pi(t)\subset\Pi$ consist of packings formed by congruent cubes of sidelength $t\le 1$. Taking in the previous inequality supremum over all $\pi\in\Pi(t)$ we obtain
\begin{equation}\label{e3.1a}
t^{-s}\sup_{\pi\in\Pi(t)}\left\{\sum_{Q\in\pi}\bigl(E_{kr}(f;Q)\bigr)^r\right\}^{\frac 1 r}\le |f|_{V_\kappa}.
\end{equation}
Due to Theorem 4 of \cite[\S\,2]{Br-71} supremum here is bounded from below by $c\omega_{kr}(f;t)$ with $c=c(k,d)>0$. This and \eqref{e3.1a} imply that
\[
\lim_{t\rightarrow\infty}t^{-k}\omega_{kr}(f;t)\le c|f|_{V_\kappa}\lim_{t\rightarrow 0}t^{s-k}=0.
\]
In turn, this gives
\[
\omega_{kr}(f;\cdot)=0,
\]
see, e.g., \cite[sec.\,\textrm{III}.3.3]{Ti-63}.

However, Theorem 2 of \cite{Br-70} asserts that
\[
\begin{array}{l}
E_{kr}(f;Q^d)\le c(k,d)\,\omega_{kr}\left(f; \frac 1 k\right),
\end{array}
\]
i.e., $E_{kr}(f;Q^d)=0$ in this case and $f\in\mathcal P_{k-1}^d$ up to Lebesgue measure zero.
\end{proof}

In the subsequent proofs we assume that $V_\kappa$ spaces are of smoothness $\le k$.

\subsection{Proof of Theorem \ref{teo2.12} (a)} 
First we define a sequence of linear operators $\{T_n\}_{n\in\N}$ acting from $L_1$ to $C^\infty$.

Let $\varphi\in C^\infty(\RR^d)$ be a function supported by the $\ell_\infty^d$ unit cube $[-1,1]^d$ such that
\[
0\le\varphi\le 1\quad {\rm and}\quad \int_{\RR^d}\varphi\,dx=1.
\]
We denote by $f^0$ the extension of $f\in L_1$ by $0$ outside $Q^d$ and then define $f_n^0$, $n\in\mathbb N$, by the formula
\begin{equation}\label{eq5.7}
f_n^0(x):=\int_{\|y\|_\infty\le 1}f^0\bigl(x-\mbox{$\frac{y}{n+1}$}\bigr)\,\varphi(y)\, dy,\quad x\in\RR^d;
\end{equation}
hereafter $\|y\|_\infty:=\max_{1\le i\le d}|y_i|$ is the $\ell_\infty^d$ norm.

Next, let $a_n:\RR^d\rightarrow\RR^d$ be the $\lambda_n$-dilation with center $c$, where $\lambda_n:=\frac{n}{n+1}$ and $c=c_{Q^d}$ is the center of $Q^d$.

Then $a_n(Q^d)$ is a subcube of $Q^d$ centered at $c_{Q^d}$ and 
\begin{equation}\label{eq5.8}
{\rm dist}_{\ell_\infty^d} (\partial Q^d, a_n(Q^d))=\frac{1}{n+1}.
\end{equation}

Now we define $f_n:=T_n(f)\in C^\infty$ by setting 
\begin{equation}\label{eq5.9}
f_n=f_n^0\circ a_n.
\end{equation}

Let $\kappa:=\{\lambda, p,q\}$ be such that
\begin{equation}\label{e3.4}
1\le p,q\le\infty\quad {\rm and}\quad s(\kappa)\le k.
\end{equation}
We show that under these assumptions the sequence $\{f_n\}_{n\in\N}$, $f\in \dot V_\kappa$, satisfies the inequality
\begin{equation}\label{e3.5}
\varlimsup_{n\rightarrow\infty}|f_n|_{V_\kappa}\le |f|_{V_\kappa}.
\end{equation}

To prove \eqref{e3.5} we need the following duality result.

\begin{Lm}\label{le3.2}
Let $f\in L_r(Q)$, $1\le r< \infty$. There is a function $g_Q(f)\in L_{r'}(Q)$ such that
\begin{equation}\label{e3.6}
E_{kr}(f;Q)=\int_Q fg_Q(f)\, dx;
\end{equation}
\begin{equation}\label{e3.7}
\|g_Q(f)\|_{L_{r'}(Q)}=1\quad {\rm and}\quad
\int_{Q} x^\alpha g_Q(f)\, dx=0,\quad \quad |\alpha|\le k-1.
\end{equation}
\end{Lm}
\begin{proof}
Since the best approximation in \eqref{e3.6} is the distance in $L_q(Q)$ from $f$ to $\mathcal P_{k-1}^d|_Q$, the Hahn-Banach theorem implies, see, e.g.,  \cite[Lemma II.3.12]{DSch-58}, existence of a linear functional $\ell\in L_r(Q)^*$ of norm $1$ orthogonal to $\mathcal P_{k-1}^d|_Q$ and such that $\ell(f)=E_{kr}(f;Q)$.
Identifying $\ell$ with the function $g_q(f)\in L_{r'}(Q)\, (=L_r(Q)^*)$ and representing $\ell(f)$ as the integral we obtain the result.
\end{proof}

Using this we first prove  the following inequality for $1\le q<\infty$ that will be then extended to $q=\infty$.
\begin{equation}\label{eq5.15}
E_{kq}(f_n;Q)\le\lambda_n^{-\frac d q}\int_{\|y\|_\infty \le 1}E_{kq}(f^0;Q_n(y))\,\varphi(y)\, dy,
\end{equation}
where the cube $Q_n(y):=a_n(Q)-(n+1)^{-1}y$ is containing in $Q^d$ by \eqref{eq5.8}.

For its proof we use Lemma \ref{le3.2} to write
\begin{equation}\label{eq5.16}
E_{kq}(f_n;Q)=\int_Q f_n g_Q\, dx,
\end{equation}
where $g_Q\in L_{q'}$, $1<q'\le \infty$, is such that
\begin{equation}\label{equat5.17}
\|g_Q\|_{L_{q'}(Q)}=1\quad {\rm and}\quad \int_Q g_Q(x) x^\alpha=0,\quad |\alpha|\le k-1.
\end{equation}
Using the definition of $f_n$, see \eqref{eq5.9}, changing the order of integration in the inequality  \eqref{eq5.16} and using the orthogonality of $g_Q$ to $\mathcal P_{k-1}^d$ we have for any $m\in\mathcal P_{k-1}^d$
\[E_{kq}(f_n;Q)\le \int_{\|y\|_\infty\le 1}\varphi(y)dy\int_Q(f^0-m)(a_n(x)-(n+1)^{-1}y)g_Q(x)\, dx.
\]
Changing variable $x$ in the inner integral by $z:=a_n(x)-(n+1)^{-1}y$, hence, changing $Q$ by $Q_n(y)$ and $dx$ by $\lambda_n^{-d}\, dz$ and then estimating the obtained integral by the H\"{o}lder inequality we have
\begin{equation}\label{eq5.17}
E_{kq}(f_n;Q)\le \lambda_n^{-\frac d q}\int_{\|y\|_\infty\le 1}\varphi(y)\, dy\left(\int_{Q_n(y)}|f^0-m|^q\, dx\right)^{\frac 1 q}.
\end{equation}
Taking here infimum over $m\in\mathcal P_{k-1}^d$ and noting that $f^0=f$ on $Q_n(y)\subset Q^d$  we prove \eqref{eq5.15} for $1\le q<\infty$.

Let us extend the just proved inequality to $q=\infty$. To this end we need the following:
\begin{Lm}\label{le3.3}
If $g\in L_\infty$ and $Q\subset Q^d$, then
\begin{equation}\label{e3.12}
\lim_{q\rightarrow\infty}E_{kq}(g;Q)=E_{k\infty}(g;Q).
\end{equation}
\end{Lm}
\begin{proof}
Let $m\in\mathcal P_{k-1}^d$ be such that
\[
\|g-m\|_{L_\infty(Q)}=E_{k\infty}(g;Q).
\]
Then for each $q\in [1,\infty)$
\[
E_{kq}(g;Q)\le \|g-m\|_{L_{q}(Q)}\le \|g-m\|_{L_\infty(Q)}=E_{k\infty}(g;Q).
\]
This shows that 
\begin{equation}\label{upper}
\varlimsup_{q\rightarrow\infty}E_{kq}(g;Q)\le E_{kq}(g;Q).
\end{equation}

Conversely, let $\{q_i\}_{i\in\N}\subset [1,\infty)$  and $m_{q_i}\in\mathcal P_{k-1}^d$ be such that
\[
\begin{array}{c}
\displaystyle
\varliminf_{q\rightarrow\infty}E_{kq}(g;Q)=\lim_{i\rightarrow\infty}E_{kq_i}(g;Q)\quad {\rm and}\medskip
\\
\displaystyle
\|g-m_{q_i}\|_{L_{q_i}(Q)}=E_{kq_i}(g;Q).
\end{array}
\]
Since ${\rm dim}\, \mathcal P_{k-1}^d<\infty$, for some $c=c(k,d,Q)>0$
\[
\|m_{q_i}\|_{L_\infty(Q)}\le c\|m_{q_i}\|_{L_{q_i}(Q)}\le 2c\|g\|_{L_{q_i}(Q)}.
\]
Hence, there are a subsequence $\{m_{q_{j}}\}_{j\in J\subset\N}$ of $\{m_{q_i}\}_{i\in\N}$ and a polynomial $\tilde m_\infty\in\mathcal P_{k-1}^d$ such that $\displaystyle \lim_{J\ni j\rightarrow\infty}m_{q_{j}}= \tilde m_{\infty}$  uniformly on $Q$. 

\noindent Therefore
\[
\lim_{J \ni j\rightarrow \infty}\|g-m_{q_{j}}\|_{L_{q_{j}}(Q)}=\|g-\tilde m_\infty\|_{L_\infty(Q)}
\]
and we have
\begin{equation}\label{lower}
\varliminf_{q\rightarrow\infty} E_{kq}(g;Q)=\lim_{J\ni j\rightarrow \infty}E_{kq_{j}}(g;Q)=
\|g-\tilde m_\infty\|_{L_\infty(Q)}\ge E_{k\infty}(g;Q).
\end{equation}

Inequalities \eqref{upper} and \eqref{lower} imply the required statement.
\end{proof}

Now if $f\in L_\infty$, then for each $q\in [1,\infty)$, $y\in Q^d$
\[
E_{kq}(f^0;Q_n(y))\le \|f\|_{L_\infty}
\]
so that using Lemma \ref{le3.3} and the Lebesgue dominated convergence theorem we derive from \eqref{eq5.15} letting $q\rightarrow\infty$
\[
\begin{array}{l}
\displaystyle
E_{k\infty}(f_n;Q)=\lim_{q\rightarrow\infty}E_{kq}(f_n;Q)\le\lim_{q\rightarrow\infty}\lambda_n^{-\frac d q}\int_{\|y\|_\infty \le 1}E_{kq}(f^0;Q_n(y))\,\varphi(y)\, dy\\
\\
\displaystyle \int_{\|y\|_\infty \le 1}\lim_{q\rightarrow\infty}E_{kq}(f^0;Q_n(y))\,\varphi(y)\, dy=\int_{\|y\|_\infty \le 1}E_{k\infty}(f^0;Q_n(y))\,\varphi(y)\, dy.
\end{array}
\]

This proves inequality \eqref{eq5.15} for $q=\infty$.\smallskip

Using  \eqref{eq5.15} let us prove inequality \eqref{e3.5}.

Let $\pi\in\Pi$ and $1\le p,q\le\infty$. Applying \eqref{eq5.15} and noting that
\[
|Q_n(y)|=\lambda_n^d |Q|,\quad Q\in\pi,
\]
we obtain
\begin{equation}\label{e3.16a}
\begin{array}{l}
\displaystyle
\gamma(\pi;f_n):= \left(\sum_{Q\in\pi}(|Q|^{-\lambda}E_{kq}(f_n;Q))^p\right)^{\frac 1 p}\\
\\
\displaystyle
\qquad\qquad\le \lambda_n^{d(\lambda-\frac1 q )}\left
(\sum_{Q\in\pi}\left(|Q_n(y)|^{-\lambda}\int_{\|y\|_\infty\le 1}E_{kq}(f_{n};Q_n(y))\varphi(y)dy\right)^p\right)^{\frac 1 p}.
\end{array}
\end{equation}
Applying to the right-hand side the Minkowski inequality and noting that $\pi(y):=\{Q_n(y)\}_{Q\in\pi}$ is also a packing we obtain
\[
\begin{array}{l}
\displaystyle
\gamma(\pi;f_n)\le \lambda_n^{d(\lambda -\frac 1 q) }\int_{\|y\|_\infty \le 1}\gamma(\pi(y);f)\varphi(y)dy\\
\\
\displaystyle \qquad\qquad\le \lambda_n^{d(\lambda -\frac 1 q) }
\sup_{\pi}\gamma(\pi;f)\int_{\|y\|\le 1}\varphi\, dy=:\lambda_n^{d(\lambda -\frac 1 q) } |f|_{V_\kappa}.
\end{array}
\]

Thus, we have from here
\[
|f_n|_{V_\kappa}:=\sup_{\pi}\gamma(\pi;f_n)\le\lambda_n^{d(\lambda-\frac 1 q)}|f|_{V_\kappa}.
\]
Passing to $n\rightarrow\infty$ and noting that $\lambda_n:=\frac{n}{n+1}\rightarrow 1$ we prove inequality \eqref{e3.5}.\smallskip

To complete the proof of assertion (a) it remains to show that conversely
\begin{equation}\label{e3.15}
|f|_{V_\kappa}\le\varliminf_{n\rightarrow\infty}|f_n|_{V_\kappa}.
\end{equation}
This will be done using assertion (b) of the theorem that is proved now.

\subsection{Proof of Theorem \ref{teo2.12}\,(b)}

First let $q<\infty$. Then given $\varepsilon>0$ for every $f\in\dot V_\kappa\subset L_q$ there is a $C^\infty$ function $g$ supported by the open cube $(0,1)^d$ such that
\begin{equation}\label{eq5.21}
\|f-v\|_q<\varepsilon
\end{equation}
Further, as in \eqref{eq5.7} we define the regularizer $v_n$, $n\in\N$, by setting
\[
v_n(x):=\int_{\|y\|_\infty\le 1}v\bigl(a_n(x)-\mbox{$\frac{y}{n+1}$}\bigr)\,\varphi(y)\, dy,\quad x\in Q^d.
\]
Then we estimate $f-f_n$ by 
\begin{equation}\label{eq5.22}
\|f-f_n\|_q\le\|f-v\|_q+\|v-v_n\|_q+\|f_n-v_n\|_q.
\end{equation}
To estimate the second term we write
\[
(v-v_n)(x)=\int_{\|y\|_\infty\le 1}\left(v(x)-v(a_n(x)-(n+1)^{-1}y)\right)\,\varphi(y)\, dy.
\]
By the mean value theorem the absolute value of the integral is bounded
from above by
 \[
 c(d)|v|_{C^1(Q^d)}\max\{\|x-a_n(x)-(n+1)^{-1}y\|_\infty\, :\, x\in Q^d,\ \|y\|_\infty\le 1\}.
 \]
 Since $a_n(x):=\frac{n}{n+1}(x-c)+c$, where $c:=c_{Q^d}$, the above maximum is bounded by \penalty-10000 $\frac{1}{2(n+1)}+\frac{1}{n+1} $. Thus, $\|v-v_n\|_q\to 0$ as $n\to\infty$ and 
 \begin{equation}\label{eq5.22a}
 \varlimsup_{n\to\infty}\|f-f_n\|_q\le\varepsilon +\varlimsup_{n\to\infty}\|f_n-v_n\|_q.
 \end{equation}
 Finally, by the Minkowski inequality following by the change of variables, cf. \eqref{eq5.17}, we have
 \[
  \begin{array}{l}
 \displaystyle
 \|f_n-v_n\|_q:=\left\{\int_{Q^d}dx\left|\int_{\|y\|_\infty\le 1}(f-v)(a_n(x)-(n+1)^{-1}y)\,\varphi(y)\, dy\right|^q\,  \right\}^{\frac 1 q}\\
 \\
 \displaystyle\qquad\qquad\quad\le\lambda_n^{-\frac d q}\int_{\|y\|_\infty\le 1}\varphi(y)\left(\int_{Q_n^d(y)}|f-v|^q dx\right)^{\frac 1 q}\, dy,
 \end{array}
 \]
 where $Q_n^d(y):=a_n(Q^d)-(n+1)^{-1}y\subset Q^d$.
 
 Replacing $Q_n^d(y)$ by $Q^d$ and letting $n$ to $\infty$ we derive from here and \eqref{eq5.21} the inequality
 \begin{equation}\label{equat5.24a}
 \varlimsup_{n\to\infty}\|f-f_n\|_q\le 2\varepsilon\to 0\quad {\rm as}\quad \varepsilon\to 0.
 \end{equation}
 This proves assertion (b) for $1<q<\infty$.\smallskip
 
 Now let $q=\infty$. We should prove that for every $f\in\dot V_\kappa\subset L_\infty$
 \begin{equation}\label{e3.20}
 \lim_{n\rightarrow\infty}\int_{Q^d}(f-f_n)g\,dx=0\quad {\rm for\ each}\quad g\in L_1.
 \end{equation}

Since the space $C_0^\infty(Q^d)$ of $C^\infty$ functions compactly supported by the interior $\mathring Q^d$ of $Q^d$ is dense in $L_1$ and the sequence $\{f_n\}_{n\in\N}$ is bounded in $L_\infty$ by \eqref{e3.5}, it suffices to prove that
\[
 \lim_{n\rightarrow\infty}\int_{Q^d}f_ng\,dx=\int_{Q^d}fg\,dx\quad {\rm for\ each}\quad g\in C_0^\infty(Q^d).
\]
Inserting the integral representation of $f_n$ given by \eqref{eq5.8} in the left-hand side of the above equation, changing the order of integration and then making change of variables we obtain
\begin{equation}\label{e3.21}
\begin{array}{l}
\displaystyle
\int_{Q^d}f_ng\,dx=\int_{Q^d} g(x)\left(\int_{\|y\|_\infty}f^0\bigl(a_n(x)-\mbox{$\frac{y}{n+1}$}\bigr)\varphi (y)\, dy\right) dx\\
\\
\displaystyle \lambda_n^{-d}\int_{\|y\|_\infty\le 1}\varphi(y)\left(\int_{\hat Q_n^d(y)}f^0(z)\,g\bigl(a_n^*(z)+\mbox{$\frac{y}{n}$} \bigr)\, dz\right)dy;
\displaystyle
\end{array}
\end{equation}
here $a_n^*$ is the $\lambda_n^{-1}$ dilation of $\RR^d$ with respect to $c=c_{Q^d}$ and $\hat Q_n^d(y)$ is the cube $a_n^*(Q^d)+\frac{y}{n}$ containing $Q^d$.

However, $\hat Q_n^d(y)$ in \eqref{e3.21} can be replaced by $Q^d$, since ${\rm supp}\, f^0=Q^d$.

Thus, we obtain from \eqref{e3.21}
\[
\int_{Q^d}f_ng\, dx=\int_{Q^d}f\hat g_n\, dx,
\]
where we set
\[
\hat g_n(x):=\lambda_n^{-d}\int_{\|y\|_\infty\le 1}g\bigl(a_n^*(x)+\mbox{$\frac y n$}\bigr)\varphi(y)\, dy.
\]
Replacing here $a_n^*(x)$ by $x$ we obtain a regularizer of $g$ denoted by $g_n\, (:=\hat g_n\circ a_n)$. Since $g\in C_0^\infty(Q^d)$, the sequence $\{g_n\}_{n\in\N}$ converges to $g$ uniformly in $x\in Q^d$.

This implies that
\begin{equation}\label{e3.22}
\lim_{n\rightarrow\infty}\int_{Q^d}f\hat g_n\, dx=\int_{Q^d} fg\, dx+\lim_{n\rightarrow\infty}\lambda_n^{-d}\int_{Q^d}(\hat g_n-g_n)f\,dx.
\end{equation}

The integral in the second summand is bounded by
\[
\begin{array}{l}
\displaystyle
\|f\|_\infty\max\left\{\left|g\bigl(a_n^*(x)+\mbox{$\frac y n$}\bigr)-g\bigl(x+\mbox{$\frac{y}{n}$}\bigr)\right|\,:\, x\in Q^d,\, \|y\|_\infty\le 1\right\}\\
\\
\displaystyle \le \|f\|_\infty c(d) |g|_{C^1(Q^d)}\max_{x\in Q^d}\|a_n^*(x)-x\|_\infty.
\end{array}
\]
Since maximum here tends to $0$ and $\lambda_n^{-d}\rightarrow 1$ as $n\rightarrow\infty$, the second summand in \eqref{e3.22} is $0$.

This proves \eqref{e3.20} and assertion (b) of the theorem.
\subsection{Proof of Theorem \ref{teo2.12}\,(a) (conclusion)} Now we prove inequality \eqref{e3.15}.

First, let $q<\infty$. Since the function $f\mapsto E_{kq}(f;Q)$, $f\in L_q$, satisfies
\[
|E_{kq}(f;Q)-E_{kq}(f_n;Q)|\le \|f-f_n\|_q,
\]
we conclude from here and assertion (b) of the theorem that
\begin{equation}\label{e3.23}
\lim_{n\rightarrow\infty} E_{kq}(f_n;Q)=E_{kq}(f;Q).
\end{equation}

Now let $\pi\in\Pi$ be a packing. Then \eqref{e3.23} implies that
\[
\gamma(\pi;f):=\left(\sum_{Q\in\pi}\left(|Q|^{-\lambda} E_{kq}(f;Q)\right)^p\right)^{\frac 1 p}=\lim_{n\rightarrow\infty}\gamma(\pi;f_n)\le\varliminf_{n\rightarrow\infty}|f_n|_{V_\kappa}.
\]
Taking here supremum over all $\pi\in\Pi$ we obtain the required inequality
\[
|f|_{V_\kappa}\le\varliminf_{n\rightarrow\infty}|f_n|_{V_\kappa}.
\]

Now let $q=\infty$. Let $J\subset\N$ be an infinite subset such that
\begin{equation}\label{e3.24a}
\lim_{J\ni j\rightarrow\infty}|f_j|_{V_\kappa}=\varliminf_{n\rightarrow\infty}|f_n|_{V_\kappa}.
\end{equation}
Since for every $Q\subset Q^d$, $j\in J$
\[
E_{k\infty}(f_j;Q)\le |Q|^\lambda |f_j|_{V_\kappa}\le \lambda_j^{d(\lambda-\frac 1 q)}|Q|^\lambda |f|_{V_\kappa},
\]
see the inequality before  \eqref{e3.15}, polynomials $m_j^Q\in\mathcal P_{k-1}^d$ satisfying
\[
E_{k\infty}(f_j;Q)=\|f_j-m_j^Q\|_{L_\infty(Q)}
\]
are uniformly bounded in $L_\infty(Q)$, see the proof of Lemma \ref{le3.3}.
Hence, there is an infinite subset $J^Q\subset J$ and a polynomial $m^Q\in\mathcal P_{k-1}^d$ such that uniformly in $x\in Q$
\[
\lim_{J^Q\ni j\rightarrow\infty}m_j^Q(x)=m^Q(x).
\]
This and \eqref{e3.20} imply that the sequence $\{(f_j-m_j^Q)|_Q\}_{j\in J^Q}$ 
weak$^*$ converges in $L_\infty(Q)$ to $(f-m^Q)|_Q$.

Since the norm of a dual Banach space is countably lower semicontinuous in the weak$^*$ topology, we then have
\begin{equation}\label{e3.24}
E_{k\infty}(f;Q)\le \|f-m^Q\|_{L_\infty(Q)}\le\varliminf_{J^Q\ni j\rightarrow\infty}\|f_j-m_j^Q\|_{L_\infty(Q)}=\varliminf_{J^Q\ni j\rightarrow\infty}E_{k\infty}(f_j;Q).
\end{equation}

Now let $\pi\in\Pi$ be a packing. Using the Cantor diagonal procedure we find an infinite subset $\displaystyle \tilde J\subset\bigcap_{Q\in\pi} J^Q\, (\subset J)$ such that \eqref{e3.24} holds for every $Q\in\pi$ with $\tilde J$ instead of $J^Q$. This and \eqref{e3.24a}, in turn, imply that
\[
\begin{array}{l}
\displaystyle \left\{\sum_{Q\in\pi}\left(|Q|^{-\lambda}E_{k\infty}(f;Q)\right)^p  \right\}^{\frac 1 p}\le \left\{\sum_{Q\in\pi}\varliminf_{\tilde J\ni j\rightarrow\infty}\left(|Q|^{-\lambda}E_{k\infty}(f_j;Q)\right)^p  \right\}^{\frac 1 p}\\
\\
\displaystyle \le \varliminf_{\tilde J\ni j\rightarrow\infty}\left\{\left(|Q|^{-\lambda}E_{k\infty}(f_j;Q)\right)^p  \right\}^{\frac 1 p}\le \varliminf_{\tilde J\ni j\rightarrow\infty}|f_j|_{V_\kappa}=\varliminf_{n\rightarrow\infty}|f_n|_{V_\kappa}.
\end{array}
\]

Taking here supremum over all $\pi\in\Pi$ we obtain \eqref{e3.15} for $q=\infty$ as well.

It remains to combine \eqref{e3.15} and \eqref{e3.5} to obtain
\[
\lim_{n\rightarrow\infty}|f_n|_{V_\kappa}=|f|_{V_\kappa}.
\]

This completes the proof of assertion (a) of Theorem \ref{teo2.12} for $s(\kappa)\le k$, see \eqref{e3.4}.

The proof of Theorem \ref{teo2.12} is complete.

\subsection{Proof of Theorem \ref{teor2.13} }

Let $\dot V_\kappa^0$ denote a linear space of functions $f\in L_q$ satisfying
\begin{equation}\label{eq5.1}
\lim_{\varepsilon\rightarrow 0}\sup_{|\pi|\le\varepsilon}\left(\sum_{Q\in\pi}\bigl(|Q|^{-\lambda}E_{kq}(f;Q)\bigr)^p\right)^{\frac 1 p}=0,
\end{equation}
where $|\pi|:=\sup_{Q\in\pi}|Q|$, and $V_\kappa^0:=\dot V_\kappa^0/\mathcal P_{k-1}^d$.

We begin with the following assertion.\medskip

{\em $V_\kappa^0$ is a closed subspace of the space $V_\kappa$.}\medskip

First we show that $V_\kappa^0$ is a linear subspace of $V_\kappa$.

Let $f\in \dot V_\kappa^0$, and $\varepsilon_0>0$ and $c:=c(\varepsilon_0)$ be such that
\begin{equation}\label{eq5.2}
\gamma(\pi;f):=\left( \sum_{Q\in\pi}\bigl(|Q|^{-\lambda}E_{kq}(f;Q)\bigr)^p \right)^{\frac 1 p}\le c<\infty
\end{equation}
for every packing $\pi$ with $|\pi|\le\varepsilon_0$.

Now, an arbitrary $\pi$ is decomposed into packings
$\pi_1,\pi_2$ such that
\[
|\pi_1|\le\varepsilon_0\quad {\rm and}\quad |Q|>\varepsilon_0\quad {\rm for\ every}\quad Q\in\pi_2.
\]
By the second condition ${\rm card}\, \pi_2\le \frac{|Q^q|}{\varepsilon_0}=\varepsilon_0^{-1}$, hence, we have
\begin{equation}\label{eq5.4}
\begin{array}{l}
\displaystyle
\gamma(\pi;f)\le \bigl(\gamma(\pi_1;f)^p+\gamma(\pi_2;f)^p\bigr)^{\frac 1 p}\le \bigl(c(\varepsilon_0)^p+\varepsilon_0^{-1-\lambda p}\max_{Q\in\pi_2} E_{kq}(f;Q)^p\bigr)^{\frac 1 p}\\
\\
\qquad\quad\ \le
c(\varepsilon_0,\lambda,p)(1+\|f\|_q).
\end{array}
\end{equation}
We conclude that
\[
|f|_{ V_\kappa}:=\sup_{\pi}\gamma(\pi;f)<\infty\quad {\rm for\ every}\quad f\in \dot V_\kappa^0,
\]
i.e., $V_\kappa^0$ is a linear subspace of $V_\kappa$.

It remains to prove closedness of $V_\kappa^0$ in $V_\kappa$. To this end we define a seminorm  $T: V_\kappa\rightarrow\RR_+$ given for $\hat f:=\{f\}+\mathcal P_{k-1}^d\in V_\kappa$ by 
\[
T(\hat f):=\lim_{\varepsilon\rightarrow 0}\sup_{|\pi|\le\varepsilon}\left(\sum_{Q\in\pi}\bigl(|Q|^{-\lambda}E_{kq}(f;Q)\bigr)^p\right)^{\frac 1 p}.
\]
Since $T(\hat f)\le \|\hat f\|_{V_\kappa}$ for all $\hat f\in V_\kappa$, seminorm $T$ is continuous on $V_\kappa$. This implies closedness of the preimage $T^{-1}(\{0\})=V_\kappa^0$ in $V_\kappa$.\smallskip

Further,  we prove that {\em under the assumptions for $\kappa:=\{\lambda,p,q\}$
\begin{equation}\label{eq5.5}
1\le p\le\infty,\quad 1\le q<\infty\quad {\rm and}\quad s(\kappa)<k
\end{equation}
the subspaces $\textsc{v}_\kappa$ and $V_\kappa^0$ coincide.}\smallskip

First, we show that $\textsc{v}_\kappa$ is a closed subspace of $V_\kappa^0$. Since $\textsc{v}_\kappa={\rm clos}(C^\infty/\mathcal P_{k-1}^d,V_\kappa)$ and $V_\kappa^0$ is closed in $V_\kappa$, it suffices to prove that $C^\infty\subset \dot V_\kappa^0$.

To this end we estimate  
$E_{kq}(f;Q)$ with $f\in C^\infty$ by the Taylor formula as follows  
\[
E_{kq}(f; Q)\le c(k,d)|Q|^{\frac k d +\frac 1 q}\max_{|\alpha|=k}\max_Q |D^\alpha f |.
\]
This implies that, see \eqref{equ1.9},  
\begin{equation}\label{e3.32}
\begin{array}{l}
\displaystyle
\gamma(\pi;f):=\left(\sum_{Q\in\pi}\bigl(|Q|^{-\lambda}E_{kq}(f;Q)\bigr)^p\right)^{\frac 1 p}\le c\left(\sum_{Q\in\pi}|Q|^{\bigl(-\lambda+\frac k d +\frac 1 q\bigr)p}\right)^{\frac 1 p}\\
\\
\displaystyle =c\left(\sum_{Q\in\pi}|Q|^{\frac{k-s(\kappa)}{d}p+1}\right)^{\frac 1 p}\le c\max_{Q\in\pi}|Q|^{\frac{k-s(\kappa)}{d}}\left(\sum_{Q\in\pi}|Q|\right)^{\frac 1 p};
\end{array}
\end{equation}
where $c=c(k,d,f):=c(k,d)\max_{|\alpha|=k}\max_Q |D^\alpha f |$.

Since $s(\kappa)<k$ and the last sum in \eqref{e3.32} is $\le 1$, we obtain that
\[
\sup_{|\pi|\le\varepsilon}\gamma(\pi;f)\le c\varepsilon^{\frac{k-s(\kappa)}{d}}\rightarrow 0\quad {\rm as}\quad \varepsilon\rightarrow 0,
\]
i.e., $f\in V_\kappa^0$ as required.
\begin{R}\label{rem5.1}
{\rm In the proof of the embedding  $\textsc{v}_\kappa\subset V_\kappa^0$, the restriction $q<\infty$ is not used.}
\end{R}
The converse embedding follows from the next result.
\begin{Lm}\label{lem5.2}
Let $f\in \dot V_\kappa^0$. There is a sequence $\{f_n\}_{n\in\N}\subset C^\infty$ such that
\begin{equation}\label{eq5.6}
|f-f_n|_{ V_\kappa}
\rightarrow 0\quad {\rm as}\quad n\rightarrow \infty.
\end{equation}
\end{Lm}
\begin{proof}
Let us show that the sequence $\{f_n\}_{n\in\N}$ given by \eqref{eq5.9} is the required one. 

Let $\varepsilon\in (0,1)$ and  $\pi\in\Pi$; we write $\pi=\pi_1\cup\pi_2$, where $\pi_1$ consists of all cubes $Q\in\pi$ of sidelengths $<\varepsilon$ and $\pi_2:=\pi\setminus \pi_1$. Then
\begin{equation}\label{eq5.11}
{\rm card}\, \pi_2\le\frac{|Q^d|}{\varepsilon^d}=\varepsilon^{-d}.
\end{equation}

Now we write using notation of \eqref{eq5.2}

\begin{equation}\label{eq5.12}
\gamma(\pi;f-f_n)\le\sum_{i=1}^2\gamma(\pi_i;f-f_n)
\end{equation}
and estimate each term of the right-hand side.

Since cubes in $\pi_2$ have sidelengths $\ge\varepsilon$ we obtain, as in  the derivation of \eqref{eq5.4}, the inequality
\begin{equation}\label{eq5.13}
\gamma(\pi_2;f-f_n)\le\varepsilon^{-(\lambda+\frac 1 p )d}\,\|f-f_n\|_{q}.
\end{equation}
Next, we prove that
\begin{equation}\label{eq5.14}
\gamma(\pi_1;f-f_n)\le \left(1+\lambda_n^{d(\lambda -\frac 1 q) }\right)\sup_{|\pi|\le\varepsilon^d}\gamma(\pi;f).
\end{equation}
To this end, first we write
\[
\gamma(\pi_1;f-f_n)\le\gamma(\pi_1;f)+\gamma(\pi_1;f_n)
\]
and note that by the definition of $\pi_1$
\begin{equation}\label{e3.37}
\gamma(\pi_1;f)\le\sup_{|\pi|\le\varepsilon^d}\gamma(\pi;f).
\end{equation}
Moreover, by \eqref{e3.16a} following the Minkowski inequality we also have
\[
\gamma(\pi_1;f_n)\le \lambda_n^{d(\lambda -\frac 1 q) }\int_{\|y\|_\infty\le 1}\gamma(\pi_1(y);f)\varphi(y)dy,
\]
where $\pi_1(y):=\{Q_n(y)\}_{Q\in\pi_1}=:a_n(\pi_1)-(n+1)^{-1}y$ is also a packing of subcubes of $Q^d$ whose sidelengths $\le\lambda_n\varepsilon<\varepsilon$.
Hence, we conclude that
\[
\gamma(\pi_1;f)\le \lambda_n^{d(\lambda -\frac 1 q) }
\left(\sup_{|\pi|\le\varepsilon^d}\gamma(\pi;f)\right)\int_{\|y\|_\infty\le 1}\varphi\, dy=\lambda_n^{d(\lambda -\frac 1 q) } \sup_{|\pi|\le\varepsilon^d}\gamma(\pi;f).
\]
This and \eqref{e3.37} complete the proof of inequality \eqref{eq5.14}.

Now inequalities \eqref{eq5.12}--\eqref{eq5.14} imply
\begin{equation}\label{eq5.19}
\gamma(\pi;f-f_n)\le \left(1+\lambda_n^{d(\lambda -\frac 1 q) }\right)\sup_{|\pi|\le\varepsilon^d}\gamma(\pi;f)+\varepsilon^{-d(\lambda +\frac 1 p )}\|f-f_n\|_{q}.
\end{equation}

Taking here 
supremum over all packings $\pi$ and then letting $n$ to $\infty$ we have
\[
\varlimsup_{n\rightarrow\infty}|f-f_n|_{ V_\kappa}=\varlimsup_{n\rightarrow\infty}\sup_{\pi}\gamma(\pi;f-f_n)\le 2\sup_{|\pi|\le\varepsilon^d}\gamma(\pi;f)+\varlimsup_{n\rightarrow\infty}\|f-f_n\|_q.
\]
Since $f\in \dot V_\kappa^0$ and $q<\infty$, the first term on the right-hand side  tends to $0$ as $\varepsilon\rightarrow 0$ by the definition of $\dot V_\kappa^0$ and the second term tends to $0$ as $n\rightarrow\infty$ by Theorem \ref{teo2.12}\,(b).

Thus,  \eqref{eq5.6} is proved.
\end{proof}

The lemma implies that $V_\kappa^0\subset {\rm clos}(C^\infty/\mathcal P_{k-1}^d, V_\kappa)=:\textsc{v}_\kappa$. Together with the converse embedding this proves coincidence of $\textsc{v}_\kappa$ and $V_\kappa^0$ under conditions \eqref{eq5.5}.

To complete the proof of Theorem \ref{teor2.13} we must show that the spaces $V_\kappa^0$ and $\textsc{v}_\kappa$ coincide for $\kappa:=\{\lambda,p,q\}$ satisfying the conditions
\begin{equation}\label{e3.40}
1\le p\le\infty,\quad \lambda\ge 0 \quad {\rm and}\quad s(\kappa)<k.
\end{equation}

To this end, note that $\textsc{v}_\kappa\subset V_\kappa^0$ for $q=\infty$, see Remark \ref{rem5.1}, and  inequality \eqref{eq5.19} is proved for all $1\le q\le\infty$.
Hence, for $q=\infty$ and $f_n$ defined by \eqref{eq5.9} with $f\in \dot V_\kappa^0$ we have
\begin{equation}\label{eq5.24}
\gamma(\pi;f-f_n)\le 2\sup_{|\pi|\le\varepsilon^d}\gamma(\pi;f)+\varepsilon^{-d(\lambda+\frac 1 p )}\|f-f_n\|_\infty.
\end{equation}
Thus, as before the converse embedding $V_\kappa^0\subset\textsc{v}_\kappa$ will be proved if we show that
\[
\lim_{n\rightarrow\infty}\|f-f_n\|_\infty=0.
\]
As in the proof of Theorem \ref{teo2.12}\,(b), see Subsection~3.3, this is derived from the following:
\begin{Lm}\label{lem5.3}
Under conditions \eqref{e3.40} for every $f\in\dot V_\kappa^0$ and $\eta\in (0,1)$ there is a function $v\in C^\infty$ such that
\begin{equation}\label{eq5.25}
\|f-v\|_\infty<\eta .
\end{equation}
\end{Lm}
\begin{proof}
Since $\lambda\ge 0$, we have for $f\in \dot V_\kappa^0$
\[
\sup_{|Q|\le\varepsilon^d} E_{k\infty}(f;Q)\le\sup_{|\pi|\le\varepsilon^d}\left(\sum_{Q\in\pi}\left(|Q|^{-\lambda} E_{k\infty}(f;Q)\right)^p\right)^{\frac 1 p}\to 0\quad {\rm as}\quad \varepsilon\to 0.
\]
Moreover, since the $k$-th difference $\Delta_h^k$, see \eqref{equa1.8}, annihilates polynomials from $\mathcal P_{k-1}^d$ and is supported by a cube of sidelength at most $k\|h\|_\infty$, the $k$-modulus of continuity, see \eqref{equat1.25} for its definition, satisfies
\begin{equation}\label{eq5.26}
\omega_{k\infty}(f;t)\le 2^k\sup_{|Q|\le (kt)^d} E_{k\infty}(f;Q)\to 0\quad {\rm as}\quad t\to 0.
\end{equation}
Due to Marchaud's inequality estimating $\omega_{1q}$ via $\omega_{kq}$, see, e.g., \cite{DVL-96}, this implies that
\begin{equation}\label{eq5.28}
\lim_{t\to 0}\omega_{1\infty}(f;t)= 0.
\end{equation}

To proceed we need the next result.\medskip

\noindent \underline{\bf Theorem} ({\em Jackson}, see, e.g., \cite{DVL-96}). For every $n\in\mathbb N$ there is a linear operator $T_n : C_\infty(0,1)\rightarrow\mathcal P_{4n}^1$ such that
\begin{equation}\label{eq5.29}
\|f-T_n\|_{C[0,1]}\le c_0\sup_{|x-y|\le n^{-1}}|f(x)-f(y)|:=\omega_{1\infty}(f;n^{-1}),
\end{equation}
where $c_0$ is a numerical constant.\medskip

It is essential that $T_n$ is an {\em integral} operator; hence, \eqref{eq5.29} can be extended to functions $f\in L_\infty(0,1)$. 

Further, let $T_n^if$, $f\in L_\infty(Q^d)$, be the result of the application of $T_n$ to the function $x_i\mapsto f(x)$, $x=(x_1,\dots, x_d)$, $0\le x_i\le 1$. By the Fubini theorem the operators $T_n^i$, $1\le i\le d$, are pairwise commute, hence, $\mathcal T_n f:=\left(\prod_{i=1}^d T_n^i\right)f$ is a polynomial of degree $4n$ in each variable $x_i$.
Moreover, by \eqref{eq5.29}
\begin{equation}\label{eq5.30}
\|f-\mathcal T_n f\|_\infty\le\sum_{i=1}^d\|f-T_n^if\|_\infty\le dc_0\omega_{1\infty}(f;n^{-1}).
\end{equation}
This and \eqref{eq5.28} complete the proof of \eqref{eq5.25}.
\end{proof}

Now \eqref{eq5.25} gives, cf. \eqref{equat5.24a}, 
\[
\varlimsup_{n\to\infty}\|f-f_n\|_\infty\le 2\eta
\]
which along with \eqref{eq5.24}  implies for $f\in\dot V_\kappa^0$ that
\[
\varlimsup_{n\to\infty}|f-f|_{V_\kappa}\le 2\sup_{|\pi|\le\varepsilon^d}\gamma(\pi;f)+\varepsilon^{-d(\lambda+\frac 1 p )}2\eta.
\]
Letting here $\eta\rightarrow 0$ and then $\varepsilon\rightarrow 0$ we conclude that $f\in {\rm clos}(C^\infty,\dot V_\kappa)$.

This proves the converse inequality
\[
V_\kappa^0\subset {\rm clos}(C^\infty/\mathcal P_{k-1}^d, V_\kappa)=:\textsc{v}_\kappa
\]
and therefore coincidence of $V_\kappa^0$ and $\textsc{v}_\kappa$ under conditions \eqref{e3.40}.

The proof of Theorem \ref{teor2.13} is complete.

\sect{Proof of Theorems \ref{prop1.4}}
\noindent (a) For the set
\begin{equation}\label{equ3.1}
\mathcal B_\kappa:=\{b_\pi\in U_\kappa^0\, :\, [b_\pi]_{p'}\le 1 \}
\end{equation}
we denote by $\widehat{\mathcal B}_\kappa\subset U_\kappa$ the closure of its convex symmetric hull, i.e.,
\[
\widehat{\mathcal B}_\kappa:={\rm clos}\left(\left\{\sum_\pi \, \lambda_\pi b_\pi\, :\, \sum_\pi\, |\lambda_\pi|\le 1\, ,\, \{b_\pi\}\subset \mathcal B_\kappa\right\},U_\kappa\right).
\]
Since 
\[
\left\|\sum_\pi \, \lambda_\pi b_\pi\right\|_{U_\kappa}\le \sum_\pi |\lambda_\pi|\, [b_\pi ]_{p'}\le 1,
\]
see \eqref{eq2.5}, we conclude that
\[
\widehat{\mathcal B}_\kappa\subset B(U_\kappa).
\]
We should prove that these  sets coincide.

Let on the contrary $B(U_\kappa)\setminus\widehat{\mathcal B}_\kappa\ne\emptyset$ and $h\in B(U_\kappa)\setminus\widehat{\mathcal B}_\kappa$ is of norm $1$. Then by the Hahn-Banach theorem there is a functional $F\in U_\kappa^*$ of norm $1$ strictly separating $\widehat{\mathcal B}_\kappa$ and $\{h\}$, that is,
\begin{equation}\label{equ3.2}
F(h)=1\qquad {\rm and}\qquad \sup_{b_{\pi}\in \mathcal B_\kappa}|F(b_\pi)|\le 1-\varepsilon
\end{equation}
for some $\varepsilon\in (0,1)$.

Further, $B(U_\kappa^0)=B(U_\kappa)\cap U_\kappa^0$ is dense and $B(U_\kappa)\setminus\widehat{\mathcal B}_\kappa$ is open in $B(U_\kappa)$. Hence, there is a sufficiently close to $h$ element $g\in B(U_\kappa^0)\setminus\widehat{\mathcal B}_\kappa$ such that
\begin{equation}\label{equ3.3}
\|g\|_{U_\kappa}<1\qquad {\rm and}\qquad 1=F(h)\ge F(g)\ge 1-\frac \varepsilon 2 .
\end{equation}
By \eqref{eq2.5} there is a representation $g=\sum_\pi b_\pi$ such that $\sum_\pi [b_\pi]_{p'}\le 1$.

Now let 
\[
b_\pi^*:=\frac{b_\pi}{[b_\pi]_{p'}}.
\]
Using again \eqref{eq2.5} we have $\|b_\pi^*\|_{U_\kappa}\le 1$; hence, 
$|F(b_\pi^*)|\le 1-\varepsilon$ by \eqref{equ3.2}. This and \eqref{equ3.3} then imply
\[
1-\frac \varepsilon 2\le |F(g)|\le \sum_\pi [b_\pi]_{p'} |F(b_\pi^*)|\le 1-\varepsilon,
\]
a contradiction.

Thus, $B(U_\kappa)=\widehat{\mathcal B}_\kappa$. \medskip

\noindent (b) We should prove separability of the space $(U_\kappa,\|\cdot\|_{U_\kappa})$ under the assumption
\begin{equation}\label{equ3.4}
1<q:=q(\kappa)\le\infty,\quad 1\le p:=p(\kappa)\le\infty .
\end{equation}
To this end we define a closed linear subspace $\hat L_{q'}$ of $L_{q'}$  by the conditions 
\begin{equation}\label{equ3.5}
\int_{Q^d} f(x)x^\alpha\, dx=0,\quad |\alpha|\le k-1,\quad f\in L_{q'}.
\end{equation}
By our definition, $U_{\kappa}^0\subset\hat L_{q'}$. Since $q'<\infty$ by \eqref{equ3.4}, the metric space $L_{q'}$ and hence  $(U_{\kappa}^0,\|\cdot\|_{q'})$ are separable. 

Further, every function $f\in \hat L_{q'}$ is a $\kappa$-chain subordinate to the packing $\pi=\{Q^d\}$. 

In fact, a function $f=c_{Q^d}\,a_{Q^d}$, where $c_{Q^d}:=\|f\|_{q'}$ and $a_{Q^d}:=f/\|f\|_{q'}$, vanishes on $\mathcal P_{k-1}^d$ and, moreover, $\|a_{Q^d}\|_{q'}=1\, (=|Q^d|^{-\lambda})$. Hence, by the definition of the seminorm of   $U_\kappa$, 
\begin{equation}\label{equ3.6}
\|f\|_{U_\kappa}\le |c_{Q^d}|=\|f\|_{q'}.
\end{equation}
In other words, the linear embedding
\begin{equation}\label{equ3.7}
\hat L_{q'}\subset U_\kappa
\end{equation} 
holds with the embedding constant $1$. In particular, if $S$ is a dense countable subset of $\hat L_{q'}$ with respect to the topology defined by norm $\|\cdot\|_{q'}$, then it is dense with respect to the topology defined by seminorm $\|\cdot\|_{U_{\kappa}}$. Since $U_\kappa^0$ is a dense subspace of $U_\kappa$, the set $S$ is dense in $U_\kappa$ as well.

This completes the proof of part (b) of the theorem.
\begin{R}\label{rem3.1}
{\rm The argument of the proof and our definition of $U_\kappa^0$ show that under the assumptions of part (b) of the theorem $U_{\kappa}^0=\hat L_{q'}$. Thus $U_{\kappa}$ is the completion of $\hat L_{q'}$ with respect to the seminorm $\|\cdot\|_{U_{\kappa}^0}$.
}
\end{R}

\noindent (c) We should prove under the assumptions
\begin{equation}\label{equ3.8}
1<q:=q(\kappa)\le\infty,\quad 1\le p:=p(\kappa)\le\infty,\quad s(\kappa)\le k,
\end{equation}
that the space $U_\kappa$ is Banach.

Since $U_\kappa$ is the completion of $(U_\kappa^0,\|\cdot\|_{U_\kappa^0})$, it suffices to prove that $\|\cdot\|_{U_\kappa^0}$ is a norm, i.e., that if $\|g\|_{U_\kappa^0}=0$ for some $g\in U_{\kappa}^0$, then $g=0$. To prove this, we first show that for every $f\in \dot V_\kappa$ and $g\in U_\kappa^0$,
\begin{equation}\label{equ3.9}
\left|\,\int_{Q^d}fg\,dx\,\right|\le|f|_{V_\kappa}\,\|g\|_{U_\kappa};
\end{equation}
the integral exists, since $\dot V_\kappa\subset L_q$ and $U_\kappa^0= \hat L_{q'}\subset L_{q'}$.

Let $m_Q\in\mathcal P_{k-1}^d$ be such that
\[
\|f-m_Q\|_{L_q(Q)}=E_{kq}(f;Q).
\]
Then for $b_\pi:=\sum_{Q\in\pi}c_Q\,a_Q$ we obtain by the definition of $\kappa$-atoms
\[
\int_{Q^d}f b_\pi \,dx=\sum_{Q\in\pi} c_Q\,\int_Q (f-m_Q)\,a_Q\,dx.
\]
Applying twice the H\"{o}lder inequality we derive from here
\[
\begin{array}{l}
\displaystyle \left|\,\int_{Q^d}f b_\pi \,dx\,\right|\le\left(\sum_{Q\in\pi} |c_Q|^{p'}\right)^{\frac{1}{p'}}\left(\sum_{Q\in\pi}\bigl(\|f-m_Q\|_{L_q(Q)}\,\|a_Q\|_{q'}\bigr)^p\right)^{\frac{1}{p}}\\
\\
\displaystyle \qquad\qquad\qquad \le [b_\pi]_{p'} \left(\sum_{Q\in\pi}\bigl(|Q|^{-\lambda}E_{kq}(f;Q)\bigr)^p\right)^{\frac 1p}\le [b_\pi]_{p'} |f|_{V_\kappa}.
\end{array}
\]
Applying this estimate to $g\in U_\kappa^0$ that can be represented as a finite sum of $\kappa$-chains $b_\pi$, and then taking infimum over all such representations we obtain that
\[
\left|\,\int_{Q^d}f g \,dx\,\right|\le\left(\inf\sum_{\pi}\,[b_\pi]_{p'}\right) |f|_{V_\kappa}=\|g\|_{U_\kappa}\, |f|_{V_\kappa},
\]
as required.

Then from \eqref{equ3.9} and the equality $\|g\|_{U_\kappa}=0$ we obtain that
\begin{equation}\label{equ3.10}
\int_{Q^d}fg\, dx=0\quad {\rm for\ all}\quad f\in \dot V_\kappa .
\end{equation}

Next, we show that $C^\infty:=C^\infty(\RR^d)|_{Q^d}\subset \dot V_\kappa$ (in particular,  \eqref{equ3.10} is valid for all $f\in C^\infty$).

Let $\varphi\in C^\infty$. By the Taylor formula we have for $Q\subset Q^d$,
\[
E_{kq}(\varphi; Q)\le c(k,d)|Q|^{\frac k d +\frac 1 q}\max_{|\alpha|=k}\max_Q |D^\alpha \varphi|\le c(k,d,\varphi)|Q|^{\frac k d +\frac 1 q}.
\]
This implies that 
\begin{equation}\label{equat3.11}
|\varphi |_{V_\kappa}:=\sup_{\pi}\left(\sum_{Q\in\pi}\bigl(|Q|^{-\lambda}E_{kq}(\varphi;Q)\bigr)^p\right)^{\frac 1 p}\le c(k,d,\varphi)\sup_{\pi}\left(\sum_{Q\in\pi}|Q|^{\bigl(-\lambda+\frac k d +\frac 1 q\bigr)p}\right)^{\frac 1 p}.
\end{equation}
Here the power of $|Q|$ equals $\frac{k-s(\kappa)}{d}\,p+1$, see \eqref{equ1.2}, and by \eqref{equ3.8} $s(\kappa)\le k$. Hence, the sum in the right-hand side is bounded from above
by $\left(\sum_{Q\in\pi}|Q|\right)^{\frac 1 p}\le |Q^d|^{\frac 1 p}=1$. Therefore $|\varphi|_{V_\kappa}<\infty$ for all $\varphi\in C^\infty$, that is,
\[
C^\infty\subset \dot V_\kappa . 
\]
Hence equality \eqref{equ3.10} is valid for every $f\in C^\infty$.

Now let first $q<\infty$, hence, $C^\infty$ is dense in $L_q$. Then \eqref{equ3.10} is true  also for all $f\in L_q$ and this implies $g=0$, as required.

Similarly, we derive that $g=0$ for $q=\infty$ using the following
\begin{Lm}\label{lem3.2}
For every $f\in L_\infty$ there exists a sequence $\{f_n\}\subset C^\infty$ such that 
\begin{equation}\label{equ3.12}
\lim_{i\rightarrow\infty}\int_{Q^d}(f-f_n)g\, dx=0\quad {\rm for\ all}\quad g\in L_1.
\end{equation}
\end{Lm}
\begin{proof}
Let $\varphi$ be a nonnegative even $C^\infty$ function on $\RR^d$ supported by the unit Euclidean ball and the $L_1$-norm $1$. Extending $f,g$ by zero to $\RR^d$ we define $f_\varepsilon, g_\varepsilon$, $\varepsilon>0$, to be convolutions of the extensions (denoted by $f_0, g_0$) with the function $\varphi_\varepsilon : x\mapsto
\varepsilon^{-d} \varphi(\frac x \varepsilon)$, $x\in\RR^d$.
Then  $f_\varepsilon, g_\varepsilon\in C^\infty(\RR^d)$ and $\|g_0-g_\varepsilon\|_1\rightarrow 0$ as $\varepsilon\rightarrow 0$ for every $g\in L_1$. Moreover,
\[
\int_{\RR^d}(f_0-f_\varepsilon)g_0\, dx=\int_{\RR^d} f_0(g_0-g_\varepsilon)\, dx
\]
for all $f\in L_\infty$ and $g\in L_1$. The absolute value of the right-hand side is bounded from above by $\|f\|_\infty\, \|g_0-g_\varepsilon\|\rightarrow 0$ as $\varepsilon\rightarrow 0$. This clearly implies \eqref{equ3.12}.
\end{proof}

\begin{R}\label{rem3.3}
{\rm Since $L_1^*=L_\infty$, this lemma in other terms means that the set $C^\infty$ is dense in $L_\infty$ in the weak$^*$ topology.
}
\end{R}
The proof of part (c) of the theorem is complete.

%

\sect{Proof of Theorem \ref{te1.11}}
We prove that under the conditions
\begin{equation}
1\le p:=p(\kappa)\le\infty,\quad 1<q:=q(\kappa)\le\infty,\quad s:=s(\kappa)\le k,
\end{equation}
the dual to the {\em Banach space} $U_\kappa$,  see Theorem \ref{prop1.4}, is isometrically isomorphic to the Banach space $V_\kappa$. 

Let us recall that the space 
\[
V_\kappa=\dot V_\kappa/\mathcal P_{k-1}^d,
\]
see Definition \ref{def1.1}, i.e., its elements are factor-classes $\{f\}+\mathcal P_{k-1}^d$, where {\em functions} $f\in L_q$ satisfy
\[
|f|_{V_\kappa}:=\sup_{\pi\in\Pi(Q^d)}\left\{\sum_{Q\in\pi}\left(|Q|^{-\lambda}E_{kq}(f;Q)\right)^p\right\}^{\frac 1 p}<\infty.
\]
Since $\mathcal P_{k-1}^d$ is the null space of $\dot V_\kappa$, the norm of a class $\hat f\in V_\kappa$ satisfies
\[
\|\hat f\|_{V_\kappa}:=\inf\{|g|_{V_\kappa}\, :\, g\in\hat f\}=|g|_{V_\kappa}
\]
for every $g\in\hat f$.

\noindent Moreover, $\int_{Q^d} f h\, dx=\int_{Q^d} g h\, dx$ for functions $f,g$ of the same class and every $h\in U_\kappa^0$, since $\mathcal P_{k-1}^d$ is orthogonal to $U_\kappa^0$.
By this reason we will use in the forthcoming proof functions in $\dot V_\kappa$ instead of their related classes in $V_\kappa$.

Now we prove the isometry $(U_\kappa)^*\equiv V_\kappa$.

Due to \eqref{equ3.6} and Remark \ref{rem3.1}, the natural embedding $E:\hat L_{q'}\hookrightarrow U_\kappa$ is of norm $\le 1$ and has dense image. Passing to the conjugate map we obtain that 
\[
E^*: U_\kappa^*\hookrightarrow \hat L_{q'}^*\equiv L_q/\mathcal P_{k-1}^d
\] 
is a linear injection of norm $\le 1$.
On the other hand, $V_\kappa$ is contained in $L_q/\mathcal P_{k-1}^d$. Let us check that ${\rm range}(E^*)$ is in $V_\kappa$ and  that the linear map $E^*:U_\kappa^*\rightarrow V_\kappa$ is of norm $\le 1$.

To this end, for $\ell\in U_\kappa^*$ we denote by $f_\ell\in L_q$ an element whose image in $L_q/\mathcal P_{k-1}^d$ coincides with $E^*(\ell)$. Then we take for every $Q\subset Q^d$ a $\kappa$-atom denoted by $\hat a_Q$ such that
\begin{equation}\label{eq4.5}
\int_Q f_\ell\hat a_Q\, dx=|Q|^{-\lambda} E_{kq}(f_\ell;Q);
\end{equation}
its existence directly follows from Lemma \ref{le3.2} and the definition of $\kappa$-atoms. 

Then for a $\kappa$-chain $\hat b_\pi$ given by $\hat b_\pi:=\sum_{Q\in\pi}c_Q\hat a_Q$ 
we get from \eqref{eq4.5}
\[
[E^*(\ell)](\hat b_\pi)=\int_{Q^d}\hat b_\pi f_\ell\, dx=\sum_{Q\in\pi}c_Q|Q|^{-\lambda} E_{kq}(f_\ell;Q).
\]
This, in turn, implies
\[
\sum_{Q\in\pi} c_Q |Q|^{-\lambda} E_{kq}(f_\ell;Q)\le \|E^*(\ell)\|_{\hat L_{q'}}\|\hat b_\pi\|_{U_\kappa}\le \|\ell\|_{U_\kappa^*}\|\hat b_\pi\|_{U_\kappa}\le \left(\sum_{Q\in\pi}|c_Q|^{p'}\right)^{\frac{1}{p'}}\|\ell\|_{U_\kappa^*}.
\]
Taking here supremum over all $(c_Q)_{Q\in\pi}$ of the $\ell_{p'}(\pi)$ norm $1$ and then supremum over all $\pi$ we conclude that
\[
|f_\ell|_{V_\kappa}:=\sup_{\pi}\left(\sum_{Q\in\pi}\left(\frac{E_{kq}(f_\ell;Q)}{|Q|^\lambda}\right)^p\right)^{\frac 1p}\le \|\ell\|_{U_\kappa^*}.
\]
Hence, $E^*(\ell)\in V_\kappa$ for every $\ell\in U_\kappa^*$ and $E^*:U_\kappa^*\hookrightarrow V_\kappa$ is a linear injection of norm $\le 1$.

Next, let us show that there is a linear injection of norm $\le 1$
\begin{equation}\label{eq4.6}
F:V_\kappa\hookrightarrow U_\kappa^*
\end{equation}
such that
\begin{equation}\label{eq4.7}
FE^*={\rm id}|_{U_\kappa^*}.
\end{equation}

Actually, let $f\in \dot V_\kappa$ and $\ell_f: U_\kappa^0\rightarrow\RR$ be a linear functional given for $g\in U_\kappa^0$ by
\begin{equation}\label{eq4.8}
\ell_f(g):=\int_{Q^d}fg\, dx.
\end{equation}
Due to \eqref{equ3.9}
\[
|\ell_f(g)|\le |f|_{V_\kappa}\|g\|_{U_\kappa}.
\]
Thus, $\ell_f$ continuously extends to a linear functional  from $U_\kappa^*$ (denoted by the same symbol) and the linear map $F:V_\kappa\rightarrow U_\kappa^*$,  $\{f\}+\mathcal P_{k-1}^d\mapsto \ell_f$, is of norm $\le 1$. Moreover, $F$  is an injection. Indeed, let $\ell_f=0$ for some $f\in \dot V_\kappa$. Since $U_\kappa^0=\hat L_{q'}$ and $\dot V_\kappa\subset L_q$, equality \eqref{eq4.8} implies that $\ell_f|_{U_\kappa^0}$ determines the trivial functional on $\hat L_{q'}^*\equiv L_q/\mathcal P_{k-1}^d$. Hence, $f\in  \mathcal P_{k-1}^d$, i.e., $f$ determines the zero element of $V_\kappa$, as required.

Further, by the definitions of $E^*$ and $F$ we have for each $h\in U_\kappa^*$ and $g\in U_\kappa^0$ 
\[
[FE^*(h)](g)=\ell_{E^*(h)}(g)=\int_{Q^d}E^*(h)g\, dx=h(E(g))=h(g).
\]
The proof of \eqref{eq4.6} and \eqref{eq4.7} is complete. 

In turn, the established results mean that ${\rm range}(E^*)=V_\kappa$, ${\rm range}(F)=U_\kappa^*$ and $F$ and $E^*$ are isometries.

Theorem \ref{te1.11} is proved.


\sect{Proof of Theorem \ref{teo1.20}}
Let $p,q,\lambda\in\kappa$ and $s(\kappa)$ be such that 
\begin{equation}\label{eq6.1}
1<p\le\infty,\quad 1<q<\infty\quad {\rm and}\quad s(\kappa)<k.
\end{equation}
We prove that under these assumptions the {\em Banach spaces $ \textsc{v}_\kappa^{*}$ and $U_\kappa$ are isometrically isomorphic.} Along with Theorem \ref{te1.11} this directly implies the {\em two stars theorem} asserting that
$ \textsc{v}_\kappa^{**}$ and $V_\kappa$ are isometrically isomorphic (see Corollary \ref{cor2.8}).

The proof of the theorem is based on main results of Subsections 6.1 and 6.2:   Propositions \ref{prop6.1}, \ref{prop6.4} and Lemma \ref{lem6.2}. Subsection 6.3 contains the concluding part of the proof. 

\subsection{} In the subsequent text we identify $U_\kappa$ with its image under the natural embedding $U_\kappa\hookrightarrow U_\kappa^{**}$. 
Moreover, identifying $V_\kappa$ and $U_\kappa^*$, see Theorem \ref{te1.11}, we regard $U_\kappa$ as a linear subspace of $V_\kappa^*\, (=(U_\kappa^*)^*)$.

Further, $\mathfrak i:\textsc{v}_\kappa\hookrightarrow V_\kappa$ is the natural embedding, cf. \eqref{equ1.7}, and $\mathfrak i^*:V_\kappa^*\rightarrow\textsc{v}_\kappa^*$ is its adjoint. 
\begin{Prop}\label{prop6.1}
(a) $\mathfrak i^*: V_\kappa^*\rightarrow \textsc{v}_\kappa^*$ is a surjective linear map of norm one such that
$\mathfrak i^*|_{U_\kappa}:U_\kappa\rightarrow \textsc{v}_\kappa^*$ is an isometry.\smallskip

\noindent (b) The image $\mathfrak i^*(B(U_\kappa))$ of the closed unit ball of $U_\kappa$ is a dense subset of the closed unit ball $B(\textsc{v}_\kappa^*)$ in the weak$^*$ topology of $\textsc{v}_\kappa^*$. 
\end{Prop}
\begin{proof}
(a) We need the following 
\begin{Lm}\label{lem6.2}
The subspace $\textsc{v}_\kappa$ is weak$^*$ dense in the space $V_\kappa\, (=U_\kappa^*)$.
\end{Lm}
\begin{proof}
It suffices for each $f\in V_\kappa$ to find a bounded sequence  $\{f_n\}_{n\in\N}\subset \textsc{v}_\kappa$ such that
\begin{equation}\label{eq6.3}
\lim_{n\rightarrow\infty} (f-f_n)(u)=0\quad {\rm for\ all}\quad u\in U_\kappa .
\end{equation}

Let $\tilde f \in\dot V_\kappa$ be such that $f=\{\tilde f\}+\mathcal P_{k-1}^d\in V_\kappa\, (:=\dot V_\kappa/\mathcal P_{k-1}^d)$.
We choose $\tilde f_n\in C^\infty\subset \dot V_\kappa$ to be the approximation of $\tilde f$ given by Theorem  \ref{teo2.12} and then define 
\[
f_n:=\{\tilde f_n\}+\mathcal P_{k-1}^d\in \textsc{v}_\kappa.
\] 

Applying to $\{\tilde f_n\}_{n\in\N}$ and $\tilde f\in V_\kappa$ inequality \eqref{e2.16} we obtain
 \begin{equation}\label{equat6.5}
\lim_{n\rightarrow\infty}|\tilde f_n|_{V_\kappa}=|\tilde f|_{V_\kappa}.
\end{equation}

This gives boundedness of the sequence $\{f_n\}_{n\in\N}$ in $\textsc{v}_\kappa$.

Since  $U_\kappa^0$ is dense in  $U_\kappa$, the latter implies that it suffices to prove \eqref{eq6.3} for $u$ being a $\kappa$-atom, say $a_Q$.

In this case, we have for any polynomial $m\in\mathcal P_{k-1}^d$
\begin{equation}\label{eq6.4}
(f-f_n)(a_Q)=\int_Q (\tilde f-\tilde f_n) a_Q\, dx=\int_Q (\tilde f-\tilde f_n-m) a_Q\, dx.
\end{equation}
We choose $m$ here such that
 \[
E_{kq}(\tilde f-\tilde f_n;Q)=\|\tilde f-\tilde f_n-m\|_{L_q(Q)}
\]
and estimate the integral in \eqref{eq6.4} by the H\"{o}lder inequality. This gives
 \[
\left|\int_Q (\tilde f-\tilde f_n) a_Q\, dx\right|\le |Q|^{-\lambda} E_{kq}(\tilde f-\tilde f_n;Q)\le |Q|^{-\lambda}\|\tilde f-\tilde f_n\|_q \rightarrow 0\quad {\rm as}\quad n\rightarrow\infty
\]
by Theorem \ref{teo2.12}\,(b).

Hence, we conclude that for the  sequence  $\{\tilde f_n\}_{n\in\N}\subset C^\infty$ and every $\kappa$-atom $a_Q$
\[
(f-f_n)(a_Q)=\int_Q ( \tilde f- \tilde f_n) a_Q\, dx\rightarrow 0\quad {\rm as}\quad n\rightarrow\infty .
\]

This completes the proof of the  lemma.
\end{proof}
Now we finish the proof of assertion (a).

By definition, $\mathfrak i^*$ maps $V_\kappa^*$ linearly to $\textsc{v}_\kappa^*$ by
\begin{equation}\label{eq6.5}
\mathfrak i^*(f^*):=f^*|_{\textsc{v}_\kappa},\quad f^*\in V_\kappa^*.
\end{equation}
Moreover, every $f^*\in \textsc{v}_\kappa$ by the Hahn-Banach theorem is extended to some element of $V_\kappa^*$ with the same norm. Hence,
$\mathfrak i^*$ is a linear surjection of norm one.

To prove that $\mathfrak i^*|_{U_\kappa}$ is an isometry, we have to show that $\|\mathfrak i^*(v)\|_{\textsc{v}_\kappa^*}=\|v\|_{U_\kappa}$ for all $v\in U_\kappa$.

In fact, let $u\in U_\kappa\setminus\{0\}$. By the Hahn-Banach theorem there exists $f\in V_\kappa$ such that $\|f\|_{V_\kappa}=1$ and $f(u)=\|u\|_{U_\kappa}$.  By Lemma \ref{lem6.2} and \eqref{equat6.5} there exists a sequence $\{f_n\}_{n\in\N}\subset \textsc{v}_\kappa$ weak$^*$ converging to $f$ such that
\[
\lim_{n\to\infty}\|f_n\|_{V_\kappa}=\|f\|_{V_\kappa}=1.
\]
These imply that
\[
\|v\|_{U_\kappa}=|f(v)|=\lim_{n\to\infty}|f_n(v)|\le\sup_{g\in B(\textsc{v}_\kappa)}|g(v)|\le \sup_{h\in B(V_\kappa)}|f(v)|\le \|v\|_{U_\kappa}.
\]
Hence,
\[
\|v\|_{U_\kappa}=\sup_{g\in B(\textsc{v}_\kappa)}|g(v)|=\sup_{g\in B(\textsc{v}_\kappa)}|(\mathfrak i^*(v))(g)|:=
\|\mathfrak i^*(v)\|_{\textsc{v}_\kappa^*},
\]
as required.

This proves that $\mathfrak i^*|_{U_\kappa}$ is an isometry and completes the proof of assertion (a) of the proposition.\smallskip

\noindent (b) By the Goldstine theorem, see, e.g., \cite[Thm.\,5.5.1]{DSch-58},
$B(U_\kappa)$ is weak$^*$ dense in $B(U_\kappa^{**})\, (=B(V_\kappa^*))$.  Since $\mathfrak i^*$ is a bounded surjective linear map of norm one, $\mathfrak i^*(B(V_\kappa^*))$ coincides with the closed unit ball $B(\textsc{v}_\kappa^*)$ of $\textsc{v}_\kappa^*$. Moreover, $\mathfrak i^*$ is weak$^*$ continuous; hence, density of $B(U_\kappa)$ in $B(V_\kappa^*)$ implies that the weak$^*$ closure of $\mathfrak i^*(B(U_\kappa))$ coincides with $B(\textsc{v}_\kappa^*)$.

Proposition \ref{prop6.1} is proved.
\end{proof}

\subsection{} In the next result, $\mathcal A_\kappa$ denotes the set of $\kappa$-atoms and $\bar{\mathcal B}_\kappa$  the closure in $U_\kappa$ of the set of $\kappa$-chains
$\mathcal B_\kappa:=\{b_\pi\in U_\kappa^0\, :\, [b_\pi]_{p'}\le 1\}$, see Definitions \ref{def2.1} and \ref{def2.2}.
\begin{Prop}\label{prop6.4}
(a) If $1<p:=p(\kappa)<\infty$, then $\mathfrak i^*(\bar{\mathcal B}_\kappa)$ is a subset of $B(\textsc{v}_\kappa^*)$ compact
in the weak$^*$ topology of $\textsc{v}_\kappa^*$.\smallskip

\noindent (b) If $1\le p:=p(\kappa)\le\infty$, then the same is true for the set $\mathfrak i^*({\mathcal A}_\kappa)$.
\end{Prop}
\begin{proof}
{\bf (a)} Since for $b_\pi\in\mathcal B_\kappa$
\[
\|b_\pi\|_{U_\kappa^0}\le [b_\pi]_{p'}\le 1,
\]
$\bar{\mathcal B}_\kappa\subset B(U_\kappa)$. Since $\mathfrak i^*$ maps $B(U_\kappa)$ in  $B(\textsc{v}_\kappa^*)$, cf. Proposition \ref{prop6.1}\,(b), this gives the embedding
\begin{equation}\label{eq6.8}
\mathfrak i^*(\bar{\mathcal B}_\kappa)\subset B(\textsc{v}_\kappa^*).
\end{equation}

Further, by the Banach-Alaoglu theorem $B(\textsc{v}_\kappa^*)$ is compact in the weak$^*$ topology of $\textsc{v}_\kappa^*$. Moreover, by separability of $\textsc{v}_\kappa$ the ball $B(\textsc{v}_\kappa^*)$ equipped with this topology is metrizable. 
Hence, to establish assertion (a) it suffices to prove that the limit of every sequence of $\mathfrak i^*(\bar{\mathcal B}_\kappa)$ converging in the weak$^*$ topology of $B(\textsc{v}_\kappa^*)$ belongs to $\mathfrak i^*(\bar{\mathcal B}_\kappa)$. In turn, since $\mathcal B_\kappa$ is dense in $\bar{\mathcal B}_\kappa$ in the norm topology of $U_\kappa\, (\subset V_\kappa^*)$ and 
$\|\mathfrak i^*\|_{V_\kappa^*\to\textsc{v}_\kappa^*}=1$, it suffices to prove this statement for sequences from the set $\mathfrak i^*(\mathcal B_\kappa)$.

Hence, we should prove the following:
\begin{Stat}\label{stat6.5}
If $\{b^n\}_{i\in\N}\subset\mathcal B_\kappa$ is such that the sequence $\{\mathfrak i^*(b^n)\}_{n\in\N}$ weak$^*$ converges in $B(\textsc{v}_\kappa^*)$, then its limit belongs to $\mathfrak i^*(\bar{\mathcal B}_\kappa)$.
\end{Stat}
\begin{proof}
Let $b^n$ has the form
\[
b^n:=\sum_{i=1}^{N(n)} c_i^n a_{Q_i^n},\quad n\in\N,
\]
where $\pi_n:=\{Q_i^n\, :\, 1\le i\le N(n)\}$ is a packing.

Without loss of generality we assume that
\begin{equation}\label{eq6.9}
|Q_{i+1}^n|\le |Q_i^n|,\quad 1\le i < N(n).
\end{equation}
Further, we extend sequences $\pi_n$, $\{c_i^n\}$ and $\{a_{Q_i^n}\}$  by setting
\begin{equation}\label{eq6.10}
Q_i^n:=\{0\},\quad c_i^n:=0,\quad a_{Q_i^n}:=0\quad {\rm for}\quad i>N(n).
\end{equation}
Hence, we write
\begin{equation}\label{eq6.11}
b^n:=\sum_{i=1}^\infty c_i^n a_{Q_i^n},\quad n\in\N.
\end{equation}
Since Statement \ref{stat6.5} suffices to prove for any infinite subsequence of \eqref{eq6.11}, we use several times the Cantor diagonal method to construct in the next lemma a suitable for the consequent proof subsequence.
\begin{Lm}\label{lem6.6}
There is an infinite subsequence $\{b^n\}_{n\in J}$, $J\subset\N$, such that for every $i\in\N$ the following is true:
\begin{itemize}
\item[(a)]
$\{Q_{i}^n\}_{n\in J}$  converges in the Hausdorff metric to a closed subcube of $Q^d$ denoted by $Q_i$; \smallskip
\item[(b)] $\{ \mathfrak i^*(a_{Q_{i}^n})\}_{n\in\N}\subset B(\textsc{v}_\kappa^*)$  converges in the weak$^*$ topology of $B(\textsc{v}_\kappa^*)$;\smallskip
\item[(c)] if the limiting cube $Q_{i}$ has a nonempty interior, then the sequence $\{a_{Q_{i}^n}\}_{n\in J}$ converges in the weak topology of $L_{q'}$ (regarded as the dual space of $L_q$, $q\in (1,\infty)$);\smallskip
\item[(d)] the sequence $\{c^n:=(c_{i}^n)_{i\in\N}\}_{n\in J}$ of vectors from $B(\ell_{p'}(\N))$ converges in the weak topology of $\ell_{p'}(\N)\, (=\ell_{p}(\N)^*, p\in (1,\infty))$ to a vector denoted  by $c\, (\in B(\ell_{p'}(\N)))$.
\end{itemize}
\end{Lm}
\begin{proof}
(a) Parameterizing the set of closed subcubes of $Q^d$ by their centers and radii and using the Bolzano-Weierstrass theorem we conclude that $\{Q_1^n\}_{n\in\N}$ contains a converging in the Hausdorff metric subsequence, say, $\{Q_1^n\}_{n\in J_1}$. In turn, $\{Q_2^n\}_{n\in J_1}$ contains a converging in this metric subsequence, say, $\{Q_2^n\}_{n\in J_2}$, $J_2\subset J_1$, etc. Setting then $n_i:=\min J_i$ and $J^a:=\{n_i\}_{i\in\N}$, we obtain the required subsequence $\{Q_i^n\}_{n\in J^a}$ converging to some closed cube $Q_i\subset Q^d$, $i\in\N$.\smallskip

\noindent (b) Since $\|a_{Q_i^n}\|_{U_\kappa}\le 1$ for all $i,n\in\N$, the sequences $\{\mathfrak i^*(a_{Q_i^n})\}_{n\in J^a}\subset B(\textsc{v}_\kappa^*)$, $i\in\N$, while this ball is compact in the (metrizable) weak$^*$ topology of $\textsc{v}_\kappa^*$. Hence, these sequences contain converging in the weak$^*$ topology subsequences and therefore applying as in (a) the Cantor diagonal process to sequences $\{\mathfrak i^*(a_{Q_i^n})\}_{n\in J^a}\subset B(\textsc{v}_\kappa^*)$, $i\in\N$, we find an infinite subset $J^b\subset J^a$ such that each sequence $\{\mathfrak i^*(a_{Q_i^n})\}_{n\in J^b}$, $i\in\N$, converges in the weak$^*$ topology of $B(\textsc{v}_\kappa^*)$.

Hence, the subsequence $\{b^n\}_{n\in J^b}$ of $\{b^n\}_{n\in\N}$ satisfies conditions (a) and (b) of the lemma.
\smallskip

\noindent (c) Now, let $I\subset\N$ be such that for each $i\in I$,
\[
\lim_{J^b\ni n\to\infty}|Q_i^n|=|Q_i|>0.
\]
Since by the definition of a $\kappa$-atom
\[
\varlimsup_{J^b\ni n\to\infty}\|a_{Q_i^n}\|_{q'}\le\lim_{J^b\ni n\to\infty}|Q_i^n|^{-\lambda}=|Q_i|^{-\lambda}<\infty,
\]
each sequence $\{a_{Q_i^n}\}_{n\in J^b}$, $i\in I$, is bounded in the reflexive (as $1<q<\infty$) space $L_{q'}$. By the Banach-Alaoglu theorem each such a sequence contains a weak converging in $L_{q'}$ subsequence. Applying then the Cantor diagonal process to the family of sequences $\{a_{Q_i^n}\}_{n\in J^b}$, $i\in I$, we find an infinite subset $J^c\subset J^b$ such that all sequences $\{a_{Q_i^n}\}_{n\in J^c}$, $i\in I$, converge in the weak topology of $L_{q'}$.

Thus the subsequence $\{b^n\}_{n\in J^c}$ of $\{b^n\}_{n\in\N}$ satisfies conditions (a)--(c) of the lemma.\smallskip

\noindent (d) Since $\|c^n\|_{p'}:=\|(c_i^n)_{i\in\N}\|_{p'}=[b^n]_{p'}\le 1$, $n\in\N$, see \eqref{eq6.10}, the sequence $\{c^n\}_{n\in J^c}\subset B(\ell_{p'}(\N))$. Moreover, $\ell_{p'}$ is reflexive as $1<p<\infty$ and therefore by the Banach-Alaoglu theorem there exists an infinite subset $J^d\subset J^c$ such that the sequence $\{c^n\}_{n\in J^d}$ weak converges to a vector, say, $c$ in $B(\ell_{p'}(\N))$.

We set $J:=J^d$. Then the subsequence $\{b^n\}_{n\in J}$ of $\{b^n\}_{n\in\N}$ satisfies the required conditions (a)--(d).
\end{proof}

Thus, from now on without loss of generality we assume that the sequence $\{b^n\}\subset\mathcal B_\kappa$ of $\kappa$-chains, see \eqref{eq6.11}, satisfies the assertions of Lemma \ref{lem6.6}. 
In particular, there are closed cubes $Q_i\subset Q^d$, $i\in\N$, such that in the Hausdorff metric
\begin{equation}\label{eq6.12}
Q_i=\lim_{n\to\infty} Q_i^n.
\end{equation}
Since for each $n\in\N$ the cubes $Q_i^n$, $i\in\N$, are nonoverlapping and their volumes form a nonincreasing  sequence, see \eqref{eq6.9}, the same is true for the family of cubes $\{Q_i\}_{i\in\N}$. Thus, for every $i\in\N$
\begin{equation}\label{eq6.13}
\mathring{Q}_i\cap\mathring{Q}_{i+1}=\emptyset\quad {\rm and}\quad |Q_i|\ge |Q_{i+1}|.
\end{equation}
Here $\mathring{S}$ stands for the interior of $S\subset\mathbb R^d$.

Now we let $N=\infty$ if $|Q_i|\ne 0$ for all $i\in\N$, otherwise,
 $N$ be the minimal element of the set of integers $n\in\Z_+$ such that
\begin{equation}\label{eq6.14}
 |Q_i|=0\quad {\rm for}\quad i> n.
\end{equation}

Then due to our assumptions, see  Lemma \ref{lem6.6}\,(c), for  $N\ne 0$ there are functions $a_i\in L_{q'}$, $1\le i< N+1$, such that in the weak topology of $L_{q'}$
\begin{equation}\label{eq6.15}
a_i=\lim_{n\to\infty}a_{Q_i^n}.
\end{equation}
The properties of these functions are presented in the next result.
\begin{Lm}\label{lem6.7}
(1) If $N\ne \infty$ and $i> N$, then in the weak$^*$ topology of $B(\textsc{v}_\kappa^*)$
\begin{equation}\label{eq6.16}
\lim_{n\to\infty} \mathfrak i^*(a_{Q_i^n})=0.
\end{equation}
(2) If $N\ne 0$ and $1\le i<N+1$, then the function $a_i$ is a $\kappa$-atom subordinate to $Q_i$.
\end{Lm}
\begin{proof}
(1) We have to prove that for each $i> N$
\begin{equation}\label{eq6.17}
\lim_{n\to\infty}\hat a_{Q_i^n}(v)=0\quad {\rm for\ every}\quad v\in\textsc{v}_\kappa,
\end{equation}
where $\hat a_{Q_i^n}$ is the image of $a_{Q_i^n}$ under the natural embedding $U_\kappa\hookrightarrow U_\kappa^{**}=V_\kappa^*$, see \eqref{eq6.5}. Since $C^\infty/\mathcal P_{k-1}^d$ is dense in $\textsc{v}_\kappa$, we can take $v\in C^\infty/\mathcal P_{k-1}^d$ in which case
\begin{equation}\label{eq6.18}
|\hat a_{Q_i^n}(v)|=\left|\int_{Q_i^n}\tilde v a_{Q_i^n}\, dx\right|\le \|a_{Q_i^n}\|_{L_{q'}(Q_i^n)}E_{kq}(\tilde v;Q_i^n),
\end{equation}
here $\tilde v\in C^\infty$ is a representative of the factor-class $v\in C^\infty/\mathcal P_{k-1}^d$.\\
Due to Definition \ref{def2.1} and the Taylor formula the right-hand side is bounded from above by
\[
|Q_i^n|^{-\lambda} c(k,d)|\tilde v|_{W_q^k(Q^d)}|Q_i^n|^{\frac k d +\frac 1 q }= c(k,d,\tilde v)|Q_i^n|^{-\lambda+\frac k d +\frac 1 q }.
\]
By the definition of $s(\kappa)$, see \eqref{equ1.9}, and  \eqref{eq6.1}
\[
-\lambda+\frac k d +\frac 1 q \ge\frac k d -\left(\lambda-\frac 1 q +\frac 1 p\right)=\frac{k-s(\kappa)}{d}>0.
\]
Hence, since $|Q_i^n|\to 0$ as $n\to\infty$ for each $i> N$, the right-hand side in \eqref{eq6.18} tends to $0$ as $n$ tends to $\infty$.

This proves \eqref{eq6.16}.\smallskip

\noindent (2) For $N\ne 0$ and each $1\le i< N+1$ the sequence $\{Q_{i}^n\}_{n\in\N}$ converges to a compact cube $Q_i\subset Q^d$ with $|Q_i|>0$ and the sequence $\{a_{Q_i^n}\}_{n\in\N}$ converges to $a_i\in L_{q'}$ in the weak topology of $L_{q'}$, see Lemma \ref{lem6.2}. Hence, by the Fatou lemma we have
\[
\|a_i\|_{q'}\le\varliminf_{n\to\infty}\|a_{Q_i^n}\|_{q'}\le \varliminf_{n\to\infty}|Q_i^n|^{-\lambda}=|Q_i|^{-\lambda}.
\]
In other words, $a_i$ satisfies the inequality
\[
\|a_i\|_{q'}\le |Q_i|^{-\lambda}.
\]
It remains to prove that
\begin{equation}\label{eq6.19}
{\rm supp}\, a_i\subset Q_i\quad {\rm and}\quad a_i\perp\mathcal P_{k-1}^d.
\end{equation}

If, on the contrary, $|{\rm supp}\, a_i\setminus Q_i|>0$, then there is a nontrivial closed cube $Q\subset Q^d\setminus Q_i$ such that
\begin{equation}\label{eq6.20}
\int_{Q}|a_i|^{q'}dx>0.
\end{equation}
However, $Q_i^n\to Q_i$ in the Hausdorff metric
as $n\to\infty$ and therefore $\mathring{Q}_i^n\cap \mathring{Q}=\emptyset$ for all sufficiently large $n$.
This and condition (c) of Lemma \ref{lem6.2} imply that
\[
0=\lim_{n\rightarrow\infty}\int_{Q^d} (f\cdot1_{Q}) a_{Q_i^n}dx=\int_{Q} f a_i\, dx;
\]
here $f\in L_q$ and $1_S$ stands for the indicator of a set $S\subset\RR^d$. 

Since $f$ is arbitrary, $a_i|_Q$ is zero in $L_{q'}(Q)$ in contradiction to \eqref{eq6.20}.

To prove the second assertion of \eqref{eq6.19} we use the fact that $a_{Q_i^n}\perp\mathcal P_{k-1}^d$ for all $i$ and $n$. Hence, due to Lemma \ref{lem6.2}\,(c) for every polynomial $m\in\mathcal P_{k-1}^d$
\[
0=\lim_{ n\rightarrow\infty}\int_{Q^d}m a_{Q_{i}^n}\,dx=\int_{Q^d}m a_{i}\,dx.
\]

Thus, $a_i$ is a $\kappa$-atom subordinate to $Q_i$ (in the sequel denoted by $a_{Q_i}$). 
\end{proof}

Now we show that for $1\le N<\infty$
\begin{equation}\label{eq6.21}
v_N:=\sum_{i=1}^{N} c_i a_{Q_i}\in\mathcal B_\kappa.
\end{equation}

In fact, by Lemma \ref{lem6.7}\,(2) and \eqref{eq6.13} $a_{Q_i}$ are $\kappa$-atoms and $\{Q_i\}_{1\le i\le N}$ is a packing. Moreover, by Lemma \ref{lem6.6}\,(d) the $\kappa$-atom $v_N$ satisfies 
\[
[v_N]_{p'}=\|c\|_{p'}\le 1,
\]
as required.

Further, for $N=\infty$ 
\begin{equation}\label{eq6.22}
v_\infty:=\sum_{i=1}^\infty c_i a_{Q_i}\in\bar{\mathcal B}_\kappa.
\end{equation}
Indeed, by Lemmas \ref{lem6.6}\,(d) and \ref{lem6.7}\,(2)
\[
\left\|\sum_{i=\ell}^m c_ia_{Q_i}\right\|_{U_\kappa}\le \left(\sum_{i=\ell}^m |c_i|^{p'}\right)^{\frac{1}{p'}}\to 0
\]
as $\ell, m\to\infty$, i.e., the series  in \eqref{eq6.22} converges in $U_\kappa$.
Moreover, its partial sums belong to $\mathcal B_\kappa$, cf. \eqref{eq6.21}, hence, $v_\infty$ belongs to the closure of $\mathcal B_\kappa$.

Setting for $N=0$
\begin{equation}\label{eq6.22a}
v_0:=0
\end{equation}
we complete the proof of Statement \ref{stat6.5} by showing that in the weak$^*$ topology of $B(\textsc{v}_\kappa^*)$
\begin{equation}\label{eq6.23}
\lim_{n\to\infty} \mathfrak i^*(b^n)=\mathfrak i^*(v_N).
\end{equation}

As in Lemma \ref{lem6.7}\,(1) it suffices to prove that for every $f\in C^\infty/\mathcal P_{k-1}^d$
\begin{equation}\label{eq6.24}
\lim_{n\to\infty} f(b^n)=f(v_N).
\end{equation}

To prove \eqref{eq6.24}, we fix $\varepsilon\in (0,1)$ and denote by $\pi^n\subset\{Q_i^n\}_{i\in\N}$ 
the packing containing all cubes of nonzero volumes. Further, we represent $\pi^n$ as the union of two packings (one of which is possibly empty)
\[
\pi_1^n:=\{Q_i^n\in\pi^n\, :\, |Q_i^n|\le\varepsilon\}\quad {\rm and}\quad \pi_2^n:=\pi^n\setminus\pi_1^n.
\]
If $\pi_1^n\ne\emptyset$, then due to inequality \eqref{e3.32}  for every $S\subset\pi_1^n$ and $f\in C^\infty/\mathcal P_{k-1}^d$
\begin{equation}\label{eq6.25}
\left|\sum_{Q_i^n\in S}c_i^n f(a_{Q_i^n})\right|\le c(k,d,f)\varepsilon^{\frac{k-s(\kappa)}{d}}.
\end{equation}
Also, if $\pi_2^n\ne\emptyset$, then for some natural number $\ell_\varepsilon(n)$ we have
\[
\pi_2^n=\{Q_1^n,\dots, Q_{\ell_\varepsilon(n)}^n\},
\]
see \eqref{eq6.9}; moreover, comparing volumes of $Q^d$ and  of the union of cubes of $\pi_2^n$ we have
\begin{equation}\label{eq6.26}
\ell_\varepsilon(n)\le\frac{1}{\varepsilon},\quad n\in\N.
\end{equation}
For $\pi_2^n=\emptyset$ we write
\begin{equation}\label{eq6.26a}
\ell_\varepsilon(n):=0.
\end{equation}

Now, we set
\[
\ell_\varepsilon:=\varlimsup_{n\to\infty}\min\{\ell_\varepsilon(n),N\}\in \bigl[0,\min\bigl\{\mbox{$\frac 1 \varepsilon $},N\bigr\}\bigr]
\]
and choose an infinite subsequence
\[
J:=\{n\in\N\, :\, \ell_\varepsilon(n)=\ell_\varepsilon\}.
\]

Next, if $N\ne 0$, then by the definition of $v_N$, see \eqref{eq6.21}--\eqref{eq6.22a}  there is some $\ell:=\ell(f,\varepsilon)\in\N$ satisfying $\ell_\varepsilon\le \ell< N+1$ such that
\begin{equation}\label{eq6.27}
\left|f(v_N)-f\left(\sum_{i=1}^{\ell} c_{i }a_{Q_i}\right)\right|<\varepsilon.
\end{equation}
In this case, for $n\in\N$ and $m\in J$ we have
\begin{equation}\label{eq6.30}
\begin{array}{l}
\displaystyle
|f(v_N)-f(b^n)|\le |f(v_N)-f(b^m)|+|f(b^m)-f(b^n)|\\
\\
\displaystyle
\le
\left|f\left(v_N-\sum_{i=1}^{\ell} c_i a_{Q_i}\right)\right|
 +\left|\sum_{i=1}^{\ell} (c_i-c_i^m) f(a_{Q_i^m})\right|+\left|\sum_{i=1}^{\ell} c_i f(a_{Q_i^m}-a_{Q_i})\right|\\
 \\
 \displaystyle +\left|f\left(b^m-\sum_{i=1}^{\ell} c_i^m a_{Q_i^m}\right)\right|+|f(b^m)-f(b^n)|=:I+II_m+III_m+IV_m+V_{m,n}.
\end{array}
\end{equation}
Here $I<\varepsilon$ due to \eqref{eq6.27}. $II_m$  and $III_m$ tend to $0$ as $m\to\infty$, since
by Lemma \ref{lem6.6}\,(b),(d)
\[
\lim_{n\to\infty} f(a_{Q_i^n}-a_i)=0\quad {\rm and}\quad \lim_{n\to\infty} (c_i-c_i^n)=0,
\]
and, moreover, due to \eqref{e3.32}
\[
\sup_i |f(a_{Q_i^n})|\le \gamma(k,d,f)\sup_i |Q_i^n|^{\frac{k-s(\kappa)}{d}}<\infty.
\]
Further, $IV_m<c(k,d,f)\varepsilon^{\frac{k-s(\kappa)}{d}}
$ by \eqref{eq6.25} because $b^m-\sum_{i=1}^{\ell} c_i^m a_{Q_i^m}=\sum_{i=\ell+1}^\infty c_i^m a_{Q_i^m}$ and all cubes $Q_{i}^m$ with $i>\ell$ belong to $\pi_1^m$. 

Finally, $V_{m,n}\to 0$ as $m,n\to\infty$ because due to the condition of Statement \ref{stat6.5} the sequence $\{f(b^n)\}_{n\in\N}$ converges.

Applying these facts to \eqref{eq6.30} we obtain that there is some $n_\varepsilon\in\N$ such that for all $n\ge n_\varepsilon$
\[
|f(v_N)-f(b^n)|< 2\varepsilon +c(d,k,f)\varepsilon^{\frac{k-s(\kappa)}{d}}.
\]
This implies that
\[
\lim_{n\to\infty}f(b^n)=f(v_N)
\]
and completes the proof of Statement \ref{stat6.5} for
$N\ne 0$, see \eqref{eq6.23}, \eqref{eq6.24}.

Finally, if $N=0$, then $v_N=0$ and  for all sufficiently large $n\in\N$ the packing $\pi_2^n=\emptyset$. Thus,  for such $n$ we have by \eqref{eq6.25}
\[
|f(v_N)-f(b^n)|\le c(d,k,f)\varepsilon^{\frac{k-s(\kappa)}{d}}
\]
which implies that $\lim_{n\to\infty}f(b^n)=0=f(v_N)$ and completes the proof of Statement \ref{stat6.5} in this case as well.

The proof of part (a) of Proposition \ref{prop6.4} is complete. Hence, if $1<p:=p(\kappa)<\infty$, then the injection $\mathfrak i^*|_{U_\kappa}: U_\kappa\rightarrow \textsc{v}_\kappa^*$ maps the set $\bar{\mathcal B}_\kappa$ in a subset of $B(\textsc{v}_\kappa^*)$ compact
in the weak$^*$ topology of $\textsc{v}_\kappa^*$.\smallskip

\noindent {\bf (b)}  In this case, we should prove compactness in the weak$^*$ topology of $\textsc{v}_\kappa^*$ of the image under $\mathfrak i^*$ of the set of $\kappa$-atoms $\mathcal A_\kappa\subset U_\kappa^0$. Similarly to Statement \ref{stat6.5} this is equivalent to
the following statement:

If $\{b^n\}_{i\in\N}\subset\mathcal A_\kappa$ is such that the sequence $\{\mathfrak i^*(b^n)\}_{n\in\N}$ weak$^*$ converges in $B(\textsc{v}_\kappa^*)$, then its limit belongs to $\mathfrak i^*(\mathcal A_\kappa)$. 

As before, we may assume without loss of generality that $\{b^n\}_{i\in\N}$ satisfies conditions (a), (b), (c) of Lemma \ref{lem6.6}. Moreover, condition (d) of the lemma is trivially fulfilled for all $1\le p'\le\infty$. Therefore as in the proof of part (a) of the proposition
the sequence $\{\mathfrak i^*(b^n)\}_{n\in\N}$ weak$^*$ converges to $\mathfrak i^*(v_N)$, see \eqref{eq6.21}--\eqref{eq6.22a}. Since here $N\subset\{0,1\}$, $v_N\in\mathcal A_\kappa$.

The proof of the proposition is complete.
\end{proof}

\subsection{} Now, we complete the proof of Theorem \ref{teo1.20}.

First, we consider the case of $p:=p(\kappa)\in (1,\infty)$. \smallskip

We begin with the following:
\begin{Lm}\label{lem6.8}
The space $\textsc{v}_\kappa$ isometrically embeds in the space  $C(\mathfrak i^*(\bar{\mathcal B}_\kappa))$ of continuous functions on the (metrizable) compact space $\mathfrak i^*(\bar{\mathcal B}_\kappa)$.
\end{Lm}
\begin{proof}
Due to Theorem \ref{prop1.4}\,(a) the symmetric convex hull of $\mathcal B_\kappa$ denoted by ${\rm sc}(\mathcal B_\kappa)$ is dense in the closed unit ball $B(U_\kappa)$ in the norm topology; hence the same is true for ${\rm sc}(\bar{\mathcal B}_\kappa)$. Moreover, the image  $\mathfrak i^*(B(U_\kappa))$ of this ball is dense in $B(\textsc{v}_\kappa^*)$ in the weak$^*$ topology of the latter closed ball. Hence, the set
\begin{equation}\label{eq6.31}
\mathfrak i^*({\rm sc}(\bar{\mathcal B}_\kappa))={\rm sc}(\mathfrak i^*(\bar{\mathcal B}_\kappa))
\end{equation}
is weak$^*$ dense in $B(\textsc{v}_\kappa^*)$.

This implies for every element $v\in\textsc{v}_\kappa$ regarded as a bounded linear functional on $\textsc{v}_\kappa^*$ the equality
\begin{equation}\label{eq6.32}
\|v\|_{\textsc{v}_\kappa}=\sup_{v^*\in \mathfrak i^*(\bar{\mathcal B}_\kappa)}|v(v^*)|.
\end{equation} 
In fact, by the Hahn-Banach theorem
\begin{equation}\label{eq6.33}
\|v\|_{\textsc{v}_\kappa}=\sup_{v^*\in B(\textsc{v}_\kappa^*)}|v(v^*)|.
\end{equation}
Since every such $v$ is continuous in the weak$^*$ topology of $\textsc{v}_\kappa^*$, we can replace $B(\textsc{v}_\kappa^*)$ by its weak$^*$ dense subset ${\rm sc}(\mathfrak i^*(\bar{\mathcal B}_\kappa))$. In turn, the latter set can be replaced by the smaller set $\mathfrak i^*(\bar{\mathcal B}_\kappa)$ because, by definition,
\[
\mathfrak i^*(\bar{\mathcal B}_\kappa)\subset {\rm sc}(\mathfrak i^*(\bar{\mathcal B}_\kappa)):=\left\{\sum_i\lambda_i v_i^*\, :\, \{v_i^*\}\subset \mathfrak i^*(\bar{\mathcal B}_\kappa),\, \{\lambda_i\}\subset\RR,\, \sum_i |\lambda_i|\le 1\right\}
\]
and the supremum of $|v(v^*)|$ over the latter set is bounded from above by
\[
\sup\left(\sum_i |\lambda_i|   \right)\cdot\max_i |v(v_i^*)|\le \max_{v^*\in \mathfrak i^*(\bar{\mathcal B}_\kappa)}|v(v^*)|,
\]
as required.

Finally, since $v|_{\mathfrak i^*(\bar{\mathcal B}_\kappa)}$ is a continuous function on $\mathfrak i^*(\bar{\mathcal B}_\kappa)$ in the weak$^*$ topology induced from $B(\textsc{v}_\kappa^*)$ and its supremum norm equals $\|v\|_{\textsc{v}_\kappa}$, see \eqref{eq6.32}, the map
\[
\textsc{v}_\kappa\ni v\mapsto v(v^*),\quad v^*\in \mathfrak i^*(\bar{\mathcal B}_\kappa),
\]
is a linear isometric embedding of $\textsc{v}_\kappa$ in $C(\mathfrak i^*(\bar{\mathcal B}_\kappa))$.
\end{proof}

Now let $v^*$ be a linear continuous functional on the space $\textsc{v}_\kappa$ regarded as the closed subspace of $C(\mathfrak i^*(\bar{\mathcal B}_\kappa))$. By the Hahn-Banach theorem $v^*$ can be extended to a linear continuous functional, say, $\hat v^*$ on the latter space with the same norm. In turn, by the Riesz representation theorem there is a {\em regular finite Borel measure} on the compact space $\mathfrak i^*(\bar{\mathcal B}_\kappa)$ denoted by $\mu_{v^*}$ that represents $\hat v^*$.

This implies that
\begin{equation}\label{eq6.34}
v(v^*)=\int_{\mathfrak i^*(\bar{\mathcal B}_\kappa)}v\, d\mu_{v^*},\quad v\in\textsc{v}_\kappa.
\end{equation}

At the next stage we exploit this measure to find a similar representation for elements of $V_\kappa$.

To this end, we use the weak$^*$ density of the subspace $\textsc{v}_\kappa$ in the space $V_\kappa$, see Lemma \ref{lem6.2}. According to this lemma,  for every $v\in V_\kappa$ there is a {\em bounded} in the $\textsc{v}_\kappa$ norm sequence $\{v_j\}_{j\in\N}\subset\textsc{v}_\kappa$ such that
\begin{equation}\label{eq6.35}
\lim_{j\to\infty}v_j(u)=v(u),\quad u\in U_\kappa.
\end{equation}
Now let $\tau: \mathfrak i^*(U_\kappa)\rightarrow U_\kappa$ be the inverse to the {\em injection} $\mathfrak i^*|_{U_\kappa}: U_\kappa\rightarrow \textsc{v}_\kappa^*$, see Proposition \ref{prop6.1}\,(a). Making the change of variable $u\to \tau(v^*)$ we derive from \eqref{eq6.35}
\begin{equation}\label{eq6.36}
\lim_{j\to\infty} v_j(v^*)=(v\circ\tau)(v^*),\quad v^*\in \mathfrak i^*(\bar{\mathcal B}_\kappa).
\end{equation}
Since linear functionals $v_j:\textsc{v}_\kappa\to\RR$ are continuous in the weak$^*$ topology defined by $\textsc{v}_\kappa^*$ their traces to $\mathfrak i^*(\bar{\mathcal B}_\kappa)$ are continuous functions in the weak$^*$ topology induced from $B(\textsc{v}_\kappa^*)$. This implies the following:
\begin{Lm}\label{lem6.9}
The function $(v\circ\tau)|_{\mathfrak i^*(\bar{\mathcal B}_\kappa)}$ is  $\mu_{v^*}$-integrable and bounded.

Moreover, a function $\phi_{v^*}:V_\kappa\to\RR$ given by
\begin{equation}\label{eq6.37}
\phi_{v^*}(v):=\int_{\mathfrak i^*(\bar{\mathcal B}_\kappa)}v\circ\tau\,d\mu_{v^*}
\end{equation}
belongs to $V_\kappa^*$.
\end{Lm}
\begin{proof}
Since $\mu_{v^*}$ is a regular Borel measure, every continuous function on $\mathfrak i^*(\bar{\mathcal B}_\kappa)$ is $\mu_{v^*}$-measurable. Moreover, pointwise limits of sequences of such functions are $\mu_{v^*}$-measurable as well. In particular, $(v\circ\tau)|_{\mathfrak i^*(\bar{\mathcal B}_\kappa)}$ being the pointwise limit of a sequence of continuous functions, see \eqref{eq6.36},  is $\mu_{v^*}$-measurable.

Further, the sequence $\{v_j\}_{j\in\N}$ is bounded in the $\textsc{v}_\kappa$-norm and $\mathfrak i^*(\bar{\mathcal B}_\kappa)\subset B(\textsc{v}_\kappa^*)$. Therefore, for every $v^*\in \mathfrak i^*(\bar{\mathcal B}_\kappa)$
\[
|(v\circ\tau)(v^*)|\le\sup_j |v_j(v^*)|\le\|v^*\|_{\textsc{v}_\kappa^*}\cdot\sup_j\|v_j\|_{\textsc{v}_\kappa}\le \sup_j\|v_j\|_{\textsc{v}_\kappa}<\infty.
\]
This implies boundedness of  $(v\circ\tau)|_{\mathfrak i^*(\bar{\mathcal B}_\kappa)}$.

Since, in turn, the measure $\mu_{v^*}$ is finite, we conclude that the integral in \eqref{eq6.37} is well-defined and the function $\phi_{v^*}$ satisfies
\[
|\phi_{v^*}(v)|\le\left(\sup_{\mathfrak i^*(\bar{\mathcal B}_\kappa)}|v\circ\tau|\right)\,|\mu_{v^*}|,\quad v\in V_\kappa.
\]

Further, since $\tau:=(\mathfrak i^*|_{U_\kappa})^{-1}$, the supremum here is bounded from above by
\[
\|v\|_{V_\kappa}\cdot \sup_{b\in\bar{\mathcal B}_\kappa}\|b\|_{U_\kappa}\le \|v\|_{V_\kappa}.
\]
Thus, $\phi_{v^*}$ is a linear continuous functional on $V_\kappa$ of 
norm $\le |\mu_{v^*}|$.
\end{proof}

At the next stage we establish  weak$^*$ continuity of  $\phi_{v^*}$ on $V_\kappa$ regarded as the dual space of $U_\kappa^*$, see Theorem \ref{te1.11}. 

To this end it suffices to show that  $\phi_{v^*}^{-1}(R)\subset V_\kappa$ is weak$^*$ closed for every closed subinterval  $R\subset\RR$. Since this preimage is convex, we can use the Krein-Smulian weak$^*$ closedness criterion, see, e.g., \cite[Thm.\,V.5.7]{DSch-58}. In our case, it asserts that 
$\phi_{v^*}^{-1}(R)$ is weak$^*$ closed iff
$B_r(0)\cap \phi_{v^*}^{-1}(R)$ is for every $r>0$; here $B_r(0):=\{v\in V_\kappa\, :\, \|v\|_{V_\kappa}\le r\}$.

In turn, since $V_\kappa=U_\kappa^*$ and $U_\kappa$ is separable, see Theorem \ref{prop1.4}\,(b),  every bounded subset of $V_\kappa$ equipped with the induced weak$^*$ topology is metrizable. Hence, weak$^*$ closedness 
of $B_r(0)\cap\phi_{v^*}^{-1}(R)$ is a consequence of the following:
\begin{Lm}\label{lem6.10}
If a sequence $\{v_j\}_{j\in\N}\subset B_r(0)\cap\phi_{v^*}^{-1}(R)$ converges to some $v\in V_\kappa$, then $v\in B_r(0)\cap\phi_{v^*}^{-1}(R)$.
\end{Lm}
\begin{proof}
Weak$^*$ convergence of $\{v_j\}_{j\in\N}$ to $v$ implies pointwise convergence of the sequence of functions $\{v_j\circ\tau |_{\mathfrak i^*(\bar{\mathcal B}_p)}\}_{j\in\N}$ to the function $v\circ\tau |_{\mathfrak i^*(\bar{\mathcal B}_p)}$. Further, the functions of this sequence are $\mu_{v^*}$-measurable and bounded by $\sup_j\|v_j\|_{V_\kappa}$, see Lemma \ref{lem6.9}. Moreover, by the assumption of Lemma \ref{lem6.10}
\begin{equation}\label{eq6.38}
\sup_j\|v_j\|_{V_\kappa}\le r\quad {\rm and}\quad \phi_{v^*}(v_j)\in R,\ \  j\in\N.
\end{equation}
Therefore, the Lebesgue pointwise convergence theorem implies
\[
\lim_{j\to\infty}\phi_{v^*}(v_j)=\int_{\mathfrak i^*(\bar{\mathcal B}_p)}\left(\lim_{j\to\infty} v_j\circ\tau\right)d\mu_{v^*}=\int_{\mathfrak i^*(\bar{\mathcal B}_p)}v\circ\tau\,d\mu_{v^*}=\phi_{v^*}(v).
\]
Since $R\subset\RR$ is closed, the limit on the left-hand side belongs to $R$, hence,
the limit point $v\in B_r(0)\cap\phi_{v^*}^{-1}(R)$ as required.
\end{proof}

Thus $\phi_{v^*}$ is a weak$^*$ continuous linear functional from $V_\kappa^*$.  By the definition of the weak$^*$ topology on $V_\kappa=U_\kappa^*$ every such functional is uniquely determined by an element of $U_\kappa$, i.e., for some $u_{v^*}\in U_\kappa$
\[
\phi_{v^*}(v)=v(\mathfrak i^*(u_{v^*})),\quad v\in V_\kappa.
\]
On the other hand, see \eqref{eq6.34}, for all $v\in\textsc{v}_\kappa$,
\[
\phi_{v^*}(v)=v(v^*).
\]
Since $\mathfrak i^*|_{U_\kappa}:U_\kappa\rightarrow \textsc{v}_\kappa$ is an isometry, see Proposition \ref{prop6.1}\,(a), $\textsc{v}_\kappa$ separates points of $U_\kappa$. Hence, these two equalities imply that
\[
v^*=\mathfrak i^*(u_{v^*}).
\]
Thus, every point $v^*\in\textsc{v}_\kappa^*$ is the image under $\mathfrak i^*$ of some point of $U_\kappa$, i.e., $\mathfrak i^*: U_\kappa\to\textsc{v}_\kappa^*$ is a surjection. Moreover, $\mathfrak i^*|_{U_\kappa}$ is also an isometry. Hence $\mathfrak i^*$ is a linear isometric isomorphism of the Banach spaces $U_\kappa$ and $\textsc{v}_\kappa^*$. 

This completes the proof of Theorem \ref{teo1.20} for $1<p<\infty$.\smallskip

Now, we consider the case of $p=\infty$.\smallskip 

The proof repeats line by line the proof of the previous case with the related set $\mathcal B_\kappa$ replaced by the set $\mathcal A_\kappa$ of all $\kappa$-atoms. In this derivation we  take into account that  $\mathfrak i^*(\mathcal A_\kappa)$ is a weak$^*$ compact subset of $B(\textsc{v}_\kappa^*)$, see Proposition \ref{prop6.4}\,(b), and that ${\rm sc}(\mathfrak i^*(\mathcal A_\kappa))$ is a weak$^*$ dense subset of $B(\textsc{v}_\kappa^*)$ because ${\rm sc}(\mathcal A_\kappa)$ is dense in $B(U_\kappa)$.
We leave the details to the readers.

The proof of the theorem is complete.
\end{proof}

\end{document}